\documentclass[fleqn,preprint,3p,a4paper]{elsarticle}

\usepackage{amssymb}
\usepackage{amsmath}
\usepackage{amsthm}
\usepackage{dcolumn}
\usepackage{endnotes}
\usepackage{tabularx}
\usepackage[matrix,arrow]{xy}
\usepackage{wasysym}

\newtheorem{theorem}{Theorem}[section]
\newtheorem{proposition}[theorem]{Proposition}
\newtheorem{lemma}[theorem]{Lemma}
\newtheorem{corollary}[theorem]{Corollary}

\theoremstyle{definition}
\newtheorem{definition}[theorem]{Definition}
\newtheorem{example}[theorem]{Example}
\newtheorem{remark}[theorem]{Remark}

\newcommand{\ir}{{\mathsf{Irr}}}

\newcommand\twoheaddownarrow{\mathord{\rotatebox[origin=c]{90}{$\twoheadleftarrow$}}}

\newcommand{\mn}{\mathbb N}

\newcommand{\cl}{{\rm cl}}

\newcommand{\ua}{\mathord{\uparrow}}
\newcommand{\da}{\mathord{\downarrow}}

\newcommand{\mk}{\mathord{\mathsf{K}}}

\begin{document}

\begin{frontmatter}



\title{On $\mathbf{K}$-reflections of Scott spaces \tnoteref{t1}}
\tnotetext[t1]{This research was supported by the National Natural Science Foundation of China (Nos. 12071199, 11661057).}

\author[X. Xu]{Xiaoquan Xu}
\ead{xiqxu2002@163.com}
\address[X. Xu]{College of Mathematics and Statistics, Minnan Normal University, Zhangzhou 363000, China}

\begin{abstract}
In this paper, for a full subcategory $\mathbf{K}$ of the category of all $T_0$ spaces with continuous mappings, we investigate the questions under what conditions the $\mathbf{K}$-reflection of a Scott space is still a Scott space and under what conditions the Scott $\mathbf{K}$-completion of a poset exists. Some necessary and sufficient conditions for the $\mathbf{K}$-reflection of a Scott space to be a Scott space and for the existence of Scott $\mathbf{K}$-completion of a poset are established, respectively. It is shown that neither the sobrification nor the well-filtered reflection of the Johnstone space is a Scott space. The $\mathbf{K}$-reflections of Alexandroff spaces and the $\mathbf{K}$-completions of posets are also discussed.
\end{abstract}





\end{frontmatter}


\section{Introduction}

In domain theory and non-Hausdorff topology, the Scott topology on posets is the most important topology, and the sober spaces, well-filtered spaces and $d$-spaces form three of the most important classes (see \cite{AJ94, redbook, Jean-2013}). The reflectivities of these spaces in $T_0$ spaces have attracted considerable attention  (see \cite{Ershov-1999, Ershov-2017, redbook, Jean-2013, Keimel-Lawson, LLW, SXXZ, wu-xi-xu-zhao-19, Wyler, XXQ1, XSXZ, ZF}).

Let $\mathbf{Top}_0$ be the category of all $T_0$ spaces with continuous mappings and $\mathbf{Sob}$ the full subcategory of $\mathbf{Top}_0$ containing all sober spaces. Denote the category of all $d$-spaces with continuous mappings and that of all well-filtered spaces respectively by $\mathbf{Top}_d$ and $\mathbf{Top}_w$. It is well-known that $\mathbf{Sob}$ is reflective in $\mathbf{Top}_0$ (see \cite{redbook, Jean-2013}). Using $d$-closures, Wyler \cite{Wyler} proved that $\mathbf{Top}_d$ is reflective in $\mathbf{Top}_0$ (see also \cite{Ershov-1999, Ershov-2017}). In \cite{Keimel-Lawson}, Keimel and Lawson proved that for a full subcategory $\mathbf{K}$ of $\mathbf{Top}_0$ containing $\mathbf{Sob}$, if $\mathbf{K}$ has certain properties, then $\mathbf{K}$ is reflective in $\mathbf{Top}_0$. They showed that $\mathbf{Top}_d$ and some other categories have such properties.

For quite a long time, it was not known whether $\mathbf{Top}_w$ is reflective in $\mathbf{Top}_0$. Recently, this problem has been positively answered by three different methods (see \cite{wu-xi-xu-zhao-19, SXXZ, XSXZ, LLW}). More generally, for an adequate and full subcategory $\mathbf{K}$ of $\mathbf{Top}_0$ containing $\mathbf{Sob}$, a direct and uniform approach to the $\mathbf{K}$-reflections of $T_0$ spaces was provided in \cite{XXQ1}. It was shown in \cite{XXQ1} that $\mathbf{Sob}$, $\mathbf{Top}_d$, $\mathbf{Top}_w$ and Keimel-Lawson categories are all adequate. Therefore, they are all reflective in $\mathbf{Top}_0$.

Directed complete posets (dcpos for short) play a fundamental role in domain theory. In \cite{ZF}, using $d$-closures, Zhao and Fan showed that for any poset $P$, the $\mathbf{DCPO}$-completion of $P$ exists. As Keimel and Lawson pointed out in \cite{Keimel-Lawson} that the $\mathbf{DCPO}$-completion of a poset $P$ is essentially the $d$-reflection of Scott space $\Sigma~\!\! P$ or, equivalently, the $d$-reflection of Scott space of a poset $P$ is still a Scott space (see also \cite[Proposition 5.12]{XSXZ}).

For a full subcategory $\mathbf{K}$ of $\mathbf{Top}_0$ containing $\mathbf{Sob}$, a natural question arises:

\vskip 0.4cm

$\mathbf{Question ~1.}$ When is the $\mathbf{K}$-reflection of a Scott space still a Scott space?

\vskip 0.4cm

 Let $P$ be a poset. We call $P$ a $\mathbf{K}$-\emph{dcpo} if the Scott space $\Sigma~\!\!P$ is a $\mathbf{K}$-space. A Scott $\mathbf{K}$-\emph{completion}, $\mathbf{K}_s$-\emph{completion} for short, of a poset $P$ is a pair $\langle \widetilde{P}, \eta\rangle$ consisting of a $\mathbf{K}$-dcpo $\widetilde{P}$ and a Scott continuous mapping $\eta :P\longrightarrow \widetilde{P}$, such that for any Scott continuous mapping $f: P\longrightarrow Q$ to a $\mathbf{K}$-dcpo $Q$, there exists a unique Scott continuous mapping $\widetilde{f} : \widetilde{P}\longrightarrow Q$ such that $\widetilde{f}\circ\eta=f$. For $\mathbf{K}=\mathbf{Top}_w$, the $\mathbf{K}_s$-completion is simply called the $\mathbf{WF}_s$-\emph{completion}.

The ordered form of Question 1 is the following:

\vskip 0.4cm

$\mathbf{Question~ 2.}$ When does the $\mathbf{K}_s$-completion of a poset exist? In particular, when does the $\mathbf{Sob}_s$-completion of a poset exist? When does the $\mathbf{WF}_s$-completion of a poset exist?

\vskip 0.4cm

This paper is mainly devoted to investigating the above two questions. Some necessary and sufficient conditions for the $\mathbf{K}$-reflection of a Scott space to be still a Scott space and for the existence of $\mathbf{K}_s$-completion of a poset are established, respectively. A few related examples and counterexamples are presented. It is shown that for a full subcategory $\mathbf{K}$ of $\mathbf{Top}_w$ containing $\mathbf{Sob}$ which is adequate and closed with respect to homeomorphisms, the $\mathbf{K}$-reflection of the Johnstone space is not a Scott space. In particular, neither the sobrification nor the well-filtered reflection of the Johnstone space is a Scott space. In the final section, the $\mathbf{K}$-reflections of Alexandroff spaces and the $\mathbf{K}$-completions of posets are discussed.

\section{Preliminary}

In this section, we briefly recall some fundamental concepts and basic results that will be used in the paper. For further details, we refer the reader to \cite{redbook, Jean-2013, XXQ1}.

For a set $X$, $|X|$ will denote the cardinality of $X$. Let $\mathbb{N}$ denote the set of all natural numbers with the usual order and $\omega=|\mn|$. The set of all subsets of $X$ is denoted by $2^X$. Let $X^{(<\omega)}=\{F\subseteq X : F \mbox{~is a finite set}\}$ and $X^{(\leqslant\omega)}=\{F\subseteq X : F \mbox{~is a countable set}\}$.

For a poset $P$ and $A\subseteq P$, let $\mathord{\downarrow}A=\{x\in P: x\leq  a \mbox{ for some } a\in A\}$ and $\mathord{\uparrow}A=\{x\in P: x\geq  a \mbox{ for some } a\in A\}$. For  $x\in P$, we write
$\mathord{\downarrow}x$ for $\mathord{\downarrow}\{x\}$ and $\mathord{\uparrow}x$ for $\mathord{\uparrow}\{x\}$. A subset $A$ is called a \emph{lower set} (resp., an \emph{upper set}) if $A=\mathord{\downarrow}A$ (resp., $A=\mathord{\uparrow}A$). Let $\mathbf{Fin}~\!P=\{\uparrow F : F\in P^{(<\omega)}\}$ and ${\rm max}(P)=\{m\in P : m~\mbox{\rm is a maximal point of~} P\}$. In the following, when $\mathbf{Fin}~\!P$ is considered as a poset, the order $\leq$ on $\mathbf{Fin}~\!P$ always means the reverse inclusion order $\supseteq$, that is, for $\ua F_1, \ua F_2\in \mathbf{Fin}~\!P$, $\ua F_1\leq \ua F_2$ if{}f $\ua F_1\supseteq \ua F_2$.

A nonempty subset $D$ of a poset $P$ is called \emph{directed} if every two elements in $D$ have an upper bound in $D$. The set of all directed sets of $P$ is denoted by $\mathcal D(P)$. A subset $I\subseteq P$ is called an \emph{ideal} of $P$ if $I$ is a directed lower subset of $P$. The poset of all ideals (with the order of set inclusion) of $P$ is denoted by $\mathrm{Id}~\! P$. $P$ is called a \emph{directed complete poset}, or \emph{dcpo} for short, provided that $\vee D$ exists in $P$ for each $D\in \mathcal D(P)$. For $x, y\in P$, we say $x$ is \emph{way below} $y$, written $x\ll y$, if for each $D\in \mathcal D(P)$ for which $\vee D$ exists, $y\leq \vee D$ implies $x\leq d$ for some $d\in D$. Let $\twoheaddownarrow x=\{u\in P : u\ll x\}$. An element $k\in P$ is called \emph{compact} iff $k\ll k$. The subset of all compact elements of $P$ is denoted by $K(P)$. $P$ is called a \emph{continuous domain}, if for each $x\in P$, $\twoheaddownarrow x$ is directed
and $x=\vee\twoheaddownarrow x$. A continuous domain which is a complete lattice is called a \emph{continuou lattice}. $P$ is called an \emph{algebraic domain}, if for each $x\in P$, $\{k\in K(P) : k\leq x\}$ is directed and $x=\vee \{k\in K(P) : k\leq x\}$. An algebraic domain which is a complete lattice is called an \emph{algebraic lattice}. It is easy to verify that every algebraic domain is a continuous domain (see, for example, \cite[Proposition 4.3]{redbook}).

As in \cite{redbook}, the \emph{upper topology} on a poset $P$, generated
by the complements of the principal ideals of $P$, is denoted by $\upsilon (P)$. The upper sets of $P$ form the (\emph{upper}) \emph{Alexandroff topology} $\gamma (P)$. The space $\Gamma~\!\!P=(P, \gamma(P))$ is called the \emph{Alexandroff space} of $P$. A subset $U$ of $P$ is \emph{Scott open} if
(i) $U=\mathord{\uparrow}U$, and (ii) for any directed subset $D$ for which $\vee D$ exists, $\vee D\in U$ implies $D\cap U\neq\emptyset$. All Scott open subsets of $P$ form a topology. This topology is called the \emph{Scott topology} on $P$ and denoted by $\sigma(P)$. The space $\Sigma~\!\! P=(P,\sigma(P))$ is called the \emph{Scott space} of $P$. For the chain $2=\{0, 1\}$ (with the order $0<1$), we have $\sigma(2)=\{\emptyset, \{1\}, \{0,1\}\}$. The space $\Sigma~\!\!2$
is well-known under the name of \emph{Sierpinski space}.

\begin{lemma}\label{Scott-cont1} (\cite[Proposition II-2.1]{redbook})) Let $P, Q$ be posets and $f : P \longrightarrow Q$. Then the following two conditions are equivalent:
\begin{enumerate}[\rm (1)]
	\item $f$ is Scott continuous, that is, $f : \Sigma~\!\! P \longrightarrow \Sigma~\!\! Q$ is continuous.
	\item For any $D\in \mathcal D(P)$ for which $\vee D$ exists, $f(\vee D)=\vee f(D)$.
\end{enumerate}
\end{lemma}

For two spaces $X$ and $Y$, we use the symbol $X\cong Y$ to represent that $X$ and $Y$ are homeomorphic. Similarly, for two posets $P$ and $Q$, the symbol $P\cong Q$ represents that $P$ and $Q$ are isomorphic. Let $\mathcal O(X)$ (resp., $\mathcal C(X)$) be the set of all open subsets (resp., closed subsets) of $X$. For $A\subseteq X$, the closure of $A$ in $X$ is denoted by $\cl_X A$ or simply by $\overline{A}$ if there is no confusion. If $X$ is a $T_0$ space, we use $\leq_X$ to denote the \emph{specialization order} on $X$: $x\leq_X y$ if{}f $x\in \overline{\{y\}}$. The poset $X$ with the specialization order is denoted by $\Omega X$ or simply by $X$ if there is no confusion. The set $\mathcal D(\Omega X)$ is shortly denoted by $\mathcal D(X)$. Define $\mathcal S_c(X)= \{\overline{\{x\}} : x\in X\}$ and $\mathcal D_c(X)=\{\overline{D} : D\in \mathcal D(X)\}$.

In the following, when a $T_0$ space $X$ is considered as a poset, the partial order always means the specialization order unless otherwise indicated. A subset $A$ of $X$ is called \emph{saturated} if $A$ equals the intersection of all open sets containing it (equivalently, $A$ is an upper set in the specialization order).

\begin{lemma}\label{KL-basiclemma}(\cite[Lemma 6.2]{Keimel-Lawson}) Let $f : X \longrightarrow Y$ be a continuous mapping of $T_0$ spaces. If $D\in \mathcal {D}(X)$ has a supremum to which it converges, then $f(D)$ is directed and has a supremum in $Y$ to which it converges, and $f(\vee D)=\vee f(D)$.
\end{lemma}

The category of $T_0$ spaces and continuous mappings is denoted by $\mathbf{Top}_0$. For a full subcategory $\mathbf{K}$ of $\mathbf{Top}_0$, the objects of $\mathbf{K}$ will be called $\mathbf{K}$-spaces. In \cite{Keimel-Lawson}, Keimel and Lawson proposed the following properties:

($\mathrm{K}_1$) Homeomorphic copies of $\mathbf{K}$-spaces are $\mathbf{K}$-spaces.

($\mathrm{K}_2$) All sober spaces are $\mathbf{K}$-spaces or, equivalently, $\mathbf{Sob}\subseteq \mathbf{K}$.

($\mathrm{K}_3$) In a sober space $S$, the intersection of any family of $\mathbf{K}$-subspaces is a $\mathbf{K}$-space.

($\mathrm{K}_4$) Continuous maps $f : S \longrightarrow T$ between sober spaces $S$ and $T$ are $\mathbf{K}$-continuous, that is, for every $\mathbf{K}$-subspace $K$ of $T$ , the inverse image $f^{-1}(K)$ is a $\mathbf{K}$-subspace of $S$.

$\mathbf{K}$ is said to be \emph{closed with respect to homeomorphisms} if $\mathbf{K}$ has ($\mathrm{K}_1$). Clearly, $\mathbf{Sob}$, $\mathbf{Top}_w$ and $\mathbf{Top}_d$ all are closed with respect to homeomorphisms. We call $\mathbf{K}$ a \emph{Keimel-Lawson category} if it satisfies ($\mathrm{K}_1$)-($\mathrm{K}_4$).

In what follows, $\mathbf{K}$ always refers to a full subcategory $\mathbf{Top}_0$ containing $\mathbf{Sob}$ which is closed with respect to homeomorphisms.

Let $X$ be a $T_{0}$ space. A $\mathbf{K}$-\emph{reflection} of $X$ is a pair $\langle \widetilde{X},\eta_X \rangle$ consisting of a $\mathbf{K}$-space $\widetilde{X}$ and a continuous mapping $\eta_X :X\rightarrow \widetilde{X}$ satisfying that for any continuous mapping $f:X\rightarrow Y$ to a $\mathbf{K}$-space, there exists a unique continuous mapping $f^{*}:\widetilde{X}\rightarrow Y$ such that $f^{*}\circ \eta_X =f$, that is, the following diagram commutes.

\begin{equation*}
\centerline{
\xymatrix{ X \ar[dr]_{f} \ar[r]^-{\eta_X}&  \widetilde{X}\ar@{.>}[d]^{f^{*}} & \\
  & Y  & &
   }}
\end{equation*}
$\mathbf{Sob}$-reflections and $\mathbf{Top}_w$-reflections are exactly sobrifications and well-filtered reflections respectively. $\mathbf{Top}_d$-reflections are simply called $d$-\emph{reflections}.

By a standard argument, $\mathbf{K}$-reflections, if they exist, are unique up to homeomorphism. We shall use $X^k$ to denote the space of the $\mathbf{K}$-reflection of $X$ if it exists.

\begin{definition}\label{K subset} (\cite[Definition 3.2]{XXQ1})
	 A subset $A$ of a $T_0$ space $X$ is called a $\mathbf{K}$-\emph{set}, provided for any continuous mapping $ f:X\longrightarrow Y$
to a $\mathbf{K}$-space $Y$, there exists a unique $y_A\in Y$ such that $\overline{f(A)}=\overline{\{y_A\}}$.
Denote by $\mathbf{K}(X)$ the set of all closed $\mathbf{K}$-sets of $X$.
\end{definition}

Obviously, a subset $A$ of a space $X$ is a $\mathbf{K}$-set if{}f $\overline{A}$ is a $\mathbf{K}$-set. For simplicity, let $\mathbf{d}(X)=\mathbf{Top}_d(X)$ and $\mathbf{WF}(X)=\mathbf{Top}_w(X)$.

\begin{lemma}\label{SKIsetrelation} (\cite[Lemma 3.3, Corollary 3.4 and Proposition 3.8]{XXQ1}) Let $X$ be a $T_0$ space. Then
\begin{enumerate}[\rm (1)]
\item $\mathbf{Sob}(X)=\ir_c(X)$.
\item $\mathcal S_c(X)\subseteq\mathbf{K}(X)\subseteq\ir_c(X)$.
\item  $\mathcal{S}_c(X)\subseteq\mathcal{D}_c(X)\subseteq \mathbf{d}(X)\subseteq\mathbf{WF}(X)\subseteq\ir_c(X)$.
\end{enumerate}
\end{lemma}

\begin{lemma}\label{K-setimage} (\cite[Lemma 3.11]{XXQ1}) Let $X,Y$ be two $T_0$ spaces. If $f:X\longrightarrow Y$ is a continuous mapping and $A\in \mathbf{K} (X)$, then $\overline{f(A)}\in \mathbf{K} (Y)$.
\end{lemma}

\begin{lemma}\label{K-sets of closed set} Let $X$ be a $T_0$ space and $A$ a closed subspace of $X$. Then $\mathbf{K}(A)\subseteq \mathbf{K}(X)$.
\end{lemma}
\begin{proof} For $B\in \mathbf{K}(A)$, we need to show that $B\in \mathbf{K}(X)$. Suppose that $ f:X\longrightarrow Y$ is a continuous mapping
from $X$ to a $\mathbf{K}$-space $Y$. Then $ f_A : A\longrightarrow Y$, $f_A(a)=f(a)$, is continuous. By $B\in \mathbf{K}(A)$, there exists a unique $y_B\in Y$ such that $\cl_{A} f_A(B)=\overline{\{y_A\}}$. Since $A\in \mathcal C(X)$, we have $\cl_X f(B)=\cl_{A}f_A(B)=\overline{\{y_A\}}$. Thus $B\in \mathbf{K}(X)$.
\end{proof}

\section{$d$-spaces, well-filtered spaces, sober spaces and $\mathbf{K}$-spaces}

In this section, we give some known and new results about sober spaces, well-filtered spaces, $d$-spaces and $\mathbf{K}$-spaces that will be used in the other sections.

A $T_0$ space $X$ is called a $d$-space (or \emph{monotone convergence space}) if $X$ (with the specialization order) is a dcpo
 and $\mathcal O(X) \subseteq \sigma(X)$ (cf. \cite{redbook, Wyler}). Let $\mathbf{Top}_d$ be the full subcategory of $\mathbf{Top}_{0}$ containing all $d$-spaces.

\begin{lemma}\label{continuous-ScottCONT-d-space} Let $P$ be a poset, $Y$ a $T_0$ space and $f : P \longrightarrow Y$. Consider the following two conditions:
\begin{enumerate}[\rm (1)]
\item $f : \Sigma~\!\! P \longrightarrow Y$ is continuous.
\item $f : \Sigma~\!\! P \longrightarrow \Sigma~\!\! Y$ is continuous.
\end{enumerate}
\noindent Then $(1) \Rightarrow (2)$. Moreover, if $Y$ is a $d$-space, then two conditions are equivalent.
\end{lemma}
\begin{proof} (1) $\Rightarrow$ (2): By Lemma \ref{Scott-cont1} and Lemma \ref{KL-basiclemma}.

(2) $\Rightarrow$ (1): Suppose that $Y$ is a $d$-space. Then $Y$ is a dcpo (with the specialization order) and $\mathcal O(Y) \subseteq \sigma(Y)$. For each $U\in \mathcal O(Y)$, since $f : \Sigma~\!\! P \longrightarrow \Sigma~\!\! Y$ is continuous, we have $f^{-1}(U)\in \mathcal O(X)$. Thus $f : \Sigma~\!\! P \longrightarrow Y$ is continuous.
\end{proof}

A nonempty subset $A$ of a $T_{0}$ space $X$ is said to be {\it irreducible} if for any $\{F_{1},F_{2}\}\subseteq \mathcal{C} (X)$, $A\subseteq F_{1}\cup F_{2}$ always implies $A\subseteq F_{1}$ or $A\subseteq F_{2}$. Denote by $\ir(X)$ (resp., $\ir_{c}(X)$) the set of all irreducible (resp., irreducible closed) subsets of $X$. The space $X$ is called \emph{sober}, if for any $A\in\ir_c(X)$, there is a unique point $x\in X$ such that $A=\overline{\{x\}}$. Let $\mathbf{Sob}$ be the full subcategory of $\mathbf{Top}_{0}$ containing all sober spaces.

The following result is well-known (cf. \cite[Corollary II-1.12]{redbook}).

\begin{proposition}\label{Continuous domain is sober} For a continuous domain $P$, $\Sigma~\!\!P$ is sober.
\end{proposition}

For a $T_0$ space $X$, we shall use $\mathord{\mathsf{K}}(X)$ to denote the set of all nonempty compact saturated subsets of $X$ and endow it with the \emph{Smyth order}, that is, for $K_1,K_2\in \mathord{\mathsf{K}}(X)$, $K_1\sqsubseteq K_2$ if{}f $K_2\subseteq K_1$. The space $X$ is called \emph{well-filtered} if for any filtered family $\mathcal{K}\subseteq \mathord{\mathsf{K}}(X)$ and any open set $U$, $\bigcap\mathcal{K}{\subseteq} U$ implies $K{\subseteq} U$ for some $K{\in}\mathcal{K}$. Let $\mathbf{Top}_w$ be the full subcategory of $\mathbf{Top}_{0}$ containing all well-filtered spaces.

We have the following implications (which can not be reversed):
\begin{center}
sobriety $\Rightarrow$ well-filteredness $\Rightarrow$ $d$-space.
\end{center}

In \cite[Corollary 3.2]{Xi-Lawson-2017} Xi and Lawson gave the following useful result.

\begin{proposition}\label{complete lattice Scott compact closed} For a complete lattice $L$, $\Sigma \!\!~L$ is well-filtered.
\end{proposition}

It is well-known that the Johnstone space $\Sigma~\!\!\mathbb{J}$ is a $d$-space but not well-filtered (see the proof of Proposition \ref{K-reflection of the Johnstone space is not a Scott space} in Section 5). For the complete lattice $L$ constructed by Isbell in \cite{Isbell}, it is well-known that the Isbell space $\Sigma~\!\!L$ is not sober. By Proposition \ref{complete lattice Scott compact closed}, $\Sigma~\!\!L$ is well-filtered.

For any topological space $X$, $\mathcal G\subseteq 2^{X}$ and $A\subseteq X$, let $\Diamond_{\mathcal G} A=\{G\in \mathcal G : G\bigcap A\neq\emptyset\}$ and $\Box_{\mathcal G} A=\{G\in \mathcal G : G\subseteq  A\}$. The symbols $\Diamond_{\mathcal G} A$ and $\Box_{\mathcal G} A$ will be simply written as $\Diamond A$  and $\Box A$ respectively if there is no confusion. The \emph{lower Vietoris topology} on $\mathcal{G}$ is the topology that has $\{\Diamond U : U\in \mathcal O(X)\}$ as a subbase, and the resulting space is denoted by $P_H(\mathcal{G})$. If $\mathcal{G}\subseteq \ir (X)$, then $\{\Diamond_{\mathcal{G}} U : U\in \mathcal O(X)\}$ is a topology on $\mathcal{G}$. The space $P_H(\mathcal{C}(X)\setminus \{\emptyset\})$ is called the \emph{Hoare power space} or \emph{lower space} of $X$ and is denoted by $P_H(X)$ for short (cf. \cite{Schalk}). Clearly, $P_H(X)=(\mathcal{C}(X)\setminus \{\emptyset\}, \upsilon(\mathcal{C}(X)\setminus \{\emptyset\}))$ and hence it is always sober (see \cite[Proposition 2.9]{XSXZ}).

\begin{remark} \label{eta continuous} Let $X$ be a $T_0$ space.
\begin{enumerate}[\rm (1)]
	\item If $\mathcal{S}_c(X)\subseteq \mathcal{G}$, then the specialization order on $P_H(\mathcal{G})$ is the order of set inclusion, and the \emph{canonical mapping} $\eta_{X}: X\longrightarrow P_H(\mathcal{G})$, given by $\eta_X(x)=\overline {\{x\}}$, is an order and topological embedding (cf. \cite{redbook, Jean-2013, Schalk}).
    \item The space $X^s=P_H(\ir_c(X))$ with the canonical mapping $\eta_{X}: X\longrightarrow X^s$ is the \emph{sobrification} of $X$ (cf. \cite{redbook, Jean-2013}).
\end{enumerate}
\end{remark}

A full subcategory $\mathbf{K}$ of  $\mathbf{Top}_0$ is said to \emph{adequate} if for any $T_0$ space $X$, $P_H(\mathbf{K}(X))$ is a $\mathbf{K}$-space. When $\mathbf{K}$ is adequate, we have the following characterization of $\mathbf{K}$-spaces by $\mathbf{K}$-sets.

\begin{lemma}\label{K-space charac by K-set} (\cite[Corollary 4.10]{XXQ1}) Let $\mathbf{K}$ be a full subcategory of $\mathbf{Top}_d$ containing $\mathbf{Sob}$ and $X$ a $T_0$ space. Suppose that $\mathbf{K}$ is adequate and closed with respect to homeomorphisms. Then the following two conditions are equivalent:
\begin{enumerate}[\rm (1)]
\item $X$ is a $\mathbf{K}$-space.
\item $\mathbf{K}(X)=\mathcal S_c(X)$.
\end{enumerate}
\end{lemma}

\begin{lemma}\label{four categories adequate}  (\cite[Proposition 5.1, Theorem 5.4, Theorem 5.14 an d Theorem 5.17]{XXQ1}) $\mathbf{Sob}$,  $\mathbf{Top}_d$ and $\mathbf{Top}_w$ all are adequate. Moreover, every Keimel-Lawson category $\mathbf{K}$ is adequate. Therefore, they all are reflective in $\mathbf{Top}_0$.
\end{lemma}

By Lemma \ref{SKIsetrelation}, Lemma \ref{K-space charac by K-set} and Lemma \ref{four categories adequate}, we get the following two corollaries.

\begin{corollary}\label{d-space charac}  For a $T_0$ space $X$, the following conditions are equivalent:
\begin{enumerate}[\rm (1)]
\item $X$ is a $d$-space.
\item $\mathcal{D}_c(X)=\mathcal S_c(X)$.
\item $\mathbf{d}(X)=\mathcal S_c(X)$.
\end{enumerate}
\end{corollary}

\begin{corollary}\label{WF space charac}  For a $T_0$ space $X$, the following two conditions are equivalent:
\begin{enumerate}[\rm (1)]
\item $X$ is well-filtered.
\item $\mathbf{WF}(X)=\mathcal S_c(X)$.
\end{enumerate}
\end{corollary}

Now we show that if $\mathbf{K}$ is adequate and closed with respect to homeomorphisms, then $\mathbf{K}$ has equalizers and the property of being a $\mathbf{K}$-space is closed-hereditary and saturated-hereditary.

\begin{proposition}\label{K-space equalizer} Suppose that $\mathbf{K}$ is adequate and closed with respect to homeomorphisms. Let $X$ be a $\mathbf{K}$-space and $Y$ a $T_0$ space. Then for any pair of continuous mappings $f, g: X \longrightarrow Y$, the equalizer
$E(f, g)=\{x\in X : f(x)=g(x)\}$ (as a subspace of $X$) is a $\mathbf{K}$-space.
\end{proposition}

\begin{proof}  For $B\in \mathbf{K}(E(f, g))$, we show that $B\in \mathbf{K}(\mathcal S_c(E(f, g)))$. By Lemma \ref{K-setimage}, $\cl_X B\in \mathbf{K} (X)$. As $X$ is a $\mathbf{K}$-space, by Lemma \ref{K-space charac by K-set}, there is $x\in X$ such that $\cl_X B=\cl_X \{x\}$. By the continuity of $f$ and $g$, we have that $\overline{\{f(x)\}}=\overline{f(\cl_X \{x\})}=\overline{f(\cl_X B)}=\overline{f(B)}=\overline{g(B)}=\overline{g(\cl_X B)}=\overline{g(\cl_X \{x\})}=\overline{\{f(x)\}}$, and consequently, $x\in E(f, g)$ since $Y$ is a $T_0$ space. It follows that $B=(cl_X B)\cap E(f, g)=(\cl_X \{x\})\cap E(f, g)=\cl_{E(f, g)} \{x\}$, and hence $B\in \mathbf{K}(\mathcal S_c(E(f, g)))$. By Lemma \ref{K-space charac by K-set}, $E(f, g)$ is a $\mathbf{K}$-space.
\end{proof}

By Lemma \ref{four categories adequate} and Proposition \ref{K-space equalizer}, we obtain the following two results.

\begin{corollary}\label{sob wf equalizer} Suppose that $\mathbf{K}\in \{\mathbf{Sob}, \mathbf{Top}_d, \mathbf{Top}_w\}$ or $\mathbf{K}$ is a Keimel-Lawson category. Then for any pair of continuous mappings $f, g: X \longrightarrow Y$, the equalizer
$E(f, g)=\{x\in X : f(x)=g(x)\}$ (as a subspace of $X$) is a $\mathbf{K}$-space.
\end{corollary}

\begin{proposition}\label{K-spaces is closed-hereditary and saturated-hereditary} Let $\mathbf{K}$ be a full subcategory of $\mathbf{Top}_d$ containing $\mathbf{Sob}$ and $X$ a $T_0$ space. Suppose that $\mathbf{K}$ is adequate and closed with respect to homeomorphisms. Then the property of being a $\mathbf{K}$-space is closed-hereditary and saturated-hereditary.
\end{proposition}

\begin{proof} Suppose $X$ is a $\mathbf{K}$-space.

 {Case 1:} $A$ is a closed subspace of $X$.

By Lemma \ref{SKIsetrelation}, Lemma \ref{K-sets of closed set} and Lemma \ref{K-space charac by K-set}, we have that $\mathcal S_c(A)\subseteq \mathbf{K}(A)\subseteq \mathbf{K}(X)=\mathcal S_c(X)$, whence $\mathbf{K}(A)=\mathcal S_c(A)$. By Lemma \ref{K-space charac by K-set}, $A$ is a $\mathbf{K}$-space.

{Case 2:}  $U$ is a saturated subspace of $X$.

For $B\in \mathbf{K}(U)$, we show that $\cl_X B\in \mathbf{K}(X)$. Suppose that $ f:X\longrightarrow Y$ is a continuous mapping
from $X$ to a $\mathbf{K}$-space $Y$. Then $ f_U : U\longrightarrow Y$, $f_U(u)=f(u)$, is continuous. By $B\in \mathbf{K}(U)$, there exists a unique $y_U\in Y$ such that $\overline{f_U(B)}=\overline{\{y_A\}}$, that is, $\overline{f(\cl_X B)}=\overline{f(B)}=\overline{\{y_A\}}$. Thus $\cl_X B\in \mathbf{K}(X)$. Since $X$ is a $\mathbf{K}$-space, by Lemma \ref{K-space charac by K-set}, there is $x\in X$ with $\cl_X B=\cl_X \{x\}$. As $B\subseteq U$ and $U=\ua_X U$, we get $x\in U$ and hence $\cl_U B=U\cap \cl_X B=U\cap \cl_X \{x\}=\cl_U \{x\}$. By Lemma \ref{K-space charac by K-set} again, $U$ is a $\mathbf{K}$-space.

Therefore, the property of being a $\mathbf{K}$-space is closed-hereditary and saturated-hereditary.
\end{proof}

By Lemma \ref{four categories adequate} and Proposition \ref{K-spaces is closed-hereditary and saturated-hereditary}, we have the following result.

\begin{corollary}\label{K-spaces is closed-hereditary and saturated-hereditary 1} Suppose that $\mathbf{K}\in \{\mathbf{Sob}, \mathbf{Top}_d, \mathbf{Top}_w\}$ or $\mathbf{K}$ is a Keimel-Lawson category. Then the property of being a $\mathbf{K}$-space is closed-hereditary and saturated-hereditary.
\end{corollary}

\begin{definition}\label{poset plus a top} For a poset $P$, let $P_{\top}=P\cup \{\top\}$ ($\top\not\in P$) denote the
poset obtained from $P$ by adjoining a largest element $\top$ (whether $P$  has one or
not).
\end{definition}

Clearly, the order on $P_{\top}$ is as follows: $x\leq y$ iff $x\leq y$ in $P$ or $y=\top$. The element $\top$ is the largest element of $P_{\top}$ (even $P$ has a largest element).

It is straightforward to verify the following result.

\begin{lemma}\label{poset plus a top remark} Let $P$ be a poset. Then
\begin{enumerate}[\rm (1)]
\item A poset $P$ is a dcpo iff $P_{\top}$ is a dcpo. If $P$ is a dcpo, then in $P_{\top}$ we have $\top \ll \top$, i.e., $\ua \top=\{\top\}\in \sigma (P_{\top})$.

\item $\zeta_P : \Sigma~\!\! P\rightarrow \Sigma~\!\!P_{\top}$, $x\mapsto x$, is continuous.
\end{enumerate}

\end{lemma}

\begin{definition}\label{X with a largest element} For a topological space $X$, select a point $\top$ such that $\top\not\in X$. Then $\mathcal C(X)\cup \{X\cup \{\top\}\}$ (as the set of all closed sets) is a topology on $X\cup\{\top\}$. The resulting space is denoted by $X_{\top}$.
\end{definition}

\begin{remark}\label{X plus top} For a topological space $X$ and a dcpo $P$, we have
\begin{enumerate}[\rm (a)]
\item $\{\top\}$ is an open set in $X_{\top}$ and hence $X$ is a closed subspace of $X_{\top}$.
 \item $\cl_{X_\top}\{\top\}=X_\top$.
 \item  $X$ is $T_0$ iff $X_\top$ is $T_0$.
 \item $(\Sigma~\!\!P)_\top=\Sigma~\!\!P_\top$.
\end{enumerate}
\end{remark}

\begin{proposition}\label{X plus top sober}
For a topological space $X$, $X$ is sober if and only if $X_{\top}$ is sober.
\end{proposition}

\begin{proof}
  Clearly, If $X_{\top}$ is sober, then $X$, as a closed subspace of $X_\top$, is also sober since sobriety is closed-hereditary (see \cite[Exercise O-5.16]{redbook} or Corollary \ref{K-spaces is closed-hereditary and saturated-hereditary 1} below).

Conversely, if $X$ is sober, then $\ir_c(X)=\{\cl_X \{x\}: x\in X\}$. Since $\mathcal C(X_\top)=\mathcal C(X)\cup \{X_\top\}$ and $X$ is a closed subspace of $X_{\top}$, we have $\ir_{c}(X_{\top})=\ir_{c}(X)\cup \{X_{\top}\}=\{\cl_{X_\top}\{x\} : x\in X_\top\}$. Thus $X_{\top}$ is sober.
\end{proof}

Similarly, we have the following result.

\begin{proposition}\label{X plus top d-space or WF}
For a topological space $X$, $X$ is a well-filtered space (resp., $d$-space) if and only if $X_{\top}$ is a well-filtered space (resp., $d$-space).
\end{proposition}
\begin{proof} If $X_{\top}$ is a well-filtered space (resp., $d$-space), then by Corollary \ref{K-spaces is closed-hereditary and saturated-hereditary 1} below, as a closed subspace of $X_{\top}$, $X$ is a well-filtered space (resp., $d$-space).

Conversely, assume that $X$ is a $d$-space. Since $\cl_{X_\top} \{\top\}=X_\top$, $\top$ is the largest element of $X_\top$ with the specialization order. For $D\in \mathcal D(X_\top)$, if $\top\not\in D$, then $D\in \mathcal D(X)$. As $X$ is a $d$-space, by Lemma \ref{d-space charac}, there is $x\in X$ such that $\cl_X D=\cl_X \{x\}$ and hence $\cl_{X_\top} D=\cl_{X_\top} \{x\}$ since $X\in \mathcal C(X_\top)$. If $\top\in D$, then $\cl_{X_\top} D=\cl_{X_\top} \{\top\}$. Thus $X_\top$ is a $d$-space. Now we assume that $X$ is a well-filtered space. Let $\{K_d:i\in I\}\subseteq\mk(X_{\top})$ be a filtered family and $U\in \mathcal O(X_{\top})$ such that $\bigcap_{i\in I}K_i\subseteq U$. Note that $\top$ is the largest element in $X$ with respect to the specialization order, so $\top\in \bigcap_{i\in I}K_i\subseteq U$. Let $V=U\setminus\{\top\}=X\setminus(X_{\top}\setminus U)$. Then $V\in\mathcal O(X)$ and $U=V\cup\{\top\}$. For each $i\in I$, let $K^*_i=K_i\setminus\{\top\}$. One can easily check that $\{K^*_i:i\in I\}\subseteq \mk(X)$ is a filtered family and $\bigcap_{i\in I}K^*_i\subseteq V$. Since $X$ is well-filtered, there exists
	$i_0\in I$ such that $K^*_{i_0}\subseteq V$, which implies that $K_{i_0}\subseteq U$. Thus $X_\top$ is well-filtered.

\end{proof}

By Remark \ref{X plus top}, Proposition \ref{X plus top sober} and Proposition \ref{X plus top d-space or WF}, we have the following corollary.

\begin{corollary}\label{Scott space plus top sober d-sapce WF} For a dcpo $P$, we have the following conclusions:
\begin{enumerate}[\rm (1)]
\item $\Sigma ~\! P$ is a $d$-space iff $\Sigma~\!\!P_\top$ is a $d$-space.
\item $\Sigma ~\! P$ is well-filtered iff $\Sigma~\!\!P_\top$ is well-filtered.
\item $\Sigma ~\! P$ is sober iff $\Sigma~\!\!P_\top$ is sober.
\end{enumerate}
\end{corollary}

\section{$\mathbf{K}$-reflections of Scott spaces}

In this section, we will give some necessary and sufficient conditions for the $\mathbf{K}$-reflection of a $T_0$ space (especially, a Scott space) to be a Scott space. A few related examples and counterexamples are presented.

\begin{lemma}\label{lemmaclosure}(\cite[Lemma 4.3]{XXQ1})
	For a $T_0$ space $X$ and $A\subseteq X$, $\overline{\eta_X(A)}=\overline{\eta_X\left(\cl_X A\right)}=\overline{\Box_{\mathbf{K}(X)} A}=\Box \cl_X A$ in $X^k=P_H(\mathbf{K}(X))$.
\end{lemma}

By Lemma \ref{SKIsetrelation}, $\{\Diamond_{\mathbf{K}(X)} U : U\in \mathcal O(X)\}$ is a topology on $\mathbf{K}(X)$. In the following, let $\eta_X : X\longrightarrow P_H(\mathbf{K}(X))$, $\eta_X(x)=\overline {\{x\}}$, be the canonical mapping from $X$ to $P_H(\mathbf{K}(X))$. It is easy to verify that  $\eta_X$ is a topological embedding. When $X=\Sigma~\!\!P$ for some poset $P$, $\eta_X$ is simply denoted by $\eta_P$.

\begin{lemma}\label{K-adequate reflective}  (\cite[Theorem 4.6]{XXQ1})  Let $\mathbf{K}$ be a full subcategory of $\mathbf{Top}_0$ and $X$ a $T_0$ space. If $P_H(\mathbf{K}(X))$ is a $\mathbf{K}$-space, then the pair $\langle X^k=P_H(\mathbf{K}(X)), \eta_X\rangle$ is a $\mathbf{K}$-reflection of $X$. More precisely, for any continuous mapping $f:X\longrightarrow Y$ to a $\mathbf{K}$-space $Y$, there exists a unique continuous mapping $f^* :P_H(\mathbf{K}(X))\longrightarrow Y$ such that $f^*\circ\eta_X=f$, that is, the following diagram commutes.
\begin{equation*}
\centerline{\xymatrix{
	X \ar[dr]_-{f} \ar[r]^-{\eta_X}
	&P_H(\mathbf{K}(X))\ar@{.>}[d]^-{f^*}\\
	&Y}}
\end{equation*}	
\noindent The unique continuous mapping $f^* :P_H(\mathbf{K}(X))\longrightarrow Y$ is defined by $f^*(A)=y_A$, where $y_A$ is the unique point of $Y$ such that $f(A)=\overline{\{y_A\}}$.
\end{lemma}

\begin{theorem}\label{K-reflection is Scott space} Let $\mathbf{K}$ be a full subcategory of $\mathbf{Top}_d$ containing $\mathbf{Sob}$ which is adequate and closed with respect to homeomorphisms. For a $T_0$ space $X$, consider the following three conditions:
\begin{enumerate}[\rm (1)]
\item $\Sigma~\!\!\mathbf{K}(X)$ is a $\mathbf{K}$-space.
\item $\eta_X^{\sigma} : X \rightarrow \Sigma~\!\!\mathbf{K}(X)$, $\eta_X^{\sigma}(x)=\overline{\{x\}}$, is continuous.
\item The $\mathbf{K}$-reflection $X^k$ of $X$ is a Scott space.
\end{enumerate}
\noindent Then $ (1)+(2)\Rightarrow (3)$, and $(3)\Rightarrow (1)$. Moreover, when conditions (1) and (2) hold, the Scott space $\Sigma~\!\!\mathbf{K}(X)$ with the canonical mapping $\eta_X^{\sigma} : X \rightarrow \Sigma~\!\!\mathbf{K}(X)$, $\eta_X^{\sigma}(x)=\overline{\{x\}}$, is a $\mathbf{K}$-reflection of $X$.
\end{theorem}

\begin{proof} (1)+(2) $\Rightarrow$ (3): Suppose that $X$ satisfies conditions (1) and (2). Then by Lemma \ref{K-adequate reflective}, the pair $\langle X^k=P_H(\mathbf{K}(X)), \eta_X\rangle$ is a $\mathbf{K}$-reflection of $X$, and there is a unique continuous mapping $(\eta_X^{\sigma})^{*}:X^k\rightarrow \Sigma~\!\!\mathbf{K}(X)$ such that $(\eta_X^{\sigma})^{*}\circ \eta_X =\eta_X^{\sigma}$, that is, the following diagram commutes.

\begin{equation*}
\centerline{
\xymatrix{ X \ar[dr]_{\eta_X^{\sigma}} \ar[r]^-{\eta_X}&  X^k\ar@{.>}[d]^{(\eta_X^{\sigma})^{*}} & \\
  & \Sigma~\!\!\mathbf{K}(X)  & &
   }}
\end{equation*}
\noindent The unique continuous mapping $(\eta_X^{\sigma})^{*}:X^k\rightarrow \Sigma~\!\!\mathbf{K}(X)$ is defined by $(\eta_X^{\sigma})^{*}(A)=B_A$ (for each $A\in \mathbf{K}(X)$), where $B_A$ is the unique element of $\mathbf{K}(X)$ such that $\cl_{\sigma (\mathbf{K}(X))}\eta_X^{\sigma}(A)=\cl_{\sigma (\mathbf{K}(X))}\{B_A\}=\downarrow_{\mathbf{K}(X)}B_A$. It follows from Lemma \ref{SKIsetrelation} (2) that $A\subseteq B_A$. On the other hand, $X^k$ is a $d$-space by $\mathbf{K}\subseteq \mathbf{Top}_d$, and hence $\mathbf{K}(X)$ (with the order of set inclusion) is a dcpo and $\mathcal O(X^k)\subseteq \sigma (\mathbf{K}(X))$. By Lemma \ref{lemmaclosure}, we have $\downarrow_{\mathbf{K}(X)}B_A=\cl_{\sigma (\mathbf{K}(X))}\eta_X^{\sigma}(A)\subseteq \cl_{X^k}\eta_X(A)=\downarrow_{\mathbf{K}(X)}A$, and consequently, $B_A\subseteq A$ by Lemma \ref{SKIsetrelation} (2). Whence $A=B_A$, that is, $(\eta_X^{\sigma})^{*}(A)=A$ for each $A\in \mathbf{K}(X)$. By the continuity of $(\eta_X^{\sigma})^{*}$, we have $\sigma (\mathbf{K}(X))\subseteq \mathcal O(X^k)$ and hence $\mathcal O(X^k)=\sigma (\mathbf{K}(X))$, proving that the Scott space $\Sigma~\!\!\mathbf{K}(X)$ with the canonical mapping $\eta_X^{\sigma} : X \rightarrow \Sigma~\!\!\mathbf{K}(X)$ is a $\mathbf{K}$-reflection $X^k$ of $X$.

(3) $\Rightarrow$ (1): By the adequateness of $\mathbf{K}$, $\langle X^k=P_H(\mathbf{K}(X)), \eta_X\rangle$ is a $\mathbf{K}$-reflection of $X$. Suppose that the $\mathbf{K}$-reflection $X^k$ of $X$ is a Scott space. Then there is a poset $P$ and a continuous mapping $\xi_X : X \rightarrow \Sigma~\!\!P$ such that $\Sigma~\!\!P$ is a $\mathbf{K}$-space and $\langle \Sigma~\!\!P, \xi_X\rangle$ is a $\mathbf{K}$-reflection of $X$. By a standard argument, $X^k$ and $\Sigma~\!\!P$ are homeomorphic, whence $\mathbf{K}(X) ~(=\Omega P_H(\mathbf{K}(X)))$ and $P ~(=\Omega \Sigma~\!\!P)$ are isomorphic. It follows that $\Sigma~\!\!\mathbf{K}(X)$ and $\Sigma~\!\!P$ are homeomorphic. Since
$\Sigma~\!\!P$ is a $\mathbf{K}$-space and $\mathbf{K}$ is closed with respect to homeomorphisms, $\Sigma~\!\!\mathbf{K}(X)$ is a $\mathbf{K}$-space.
\end{proof}

In particular, we have the following result for the Scott space of a poset.

\begin{theorem}\label{K-reflection is Scott space 2} Let $\mathbf{K}$ be a full subcategory of $\mathbf{Top}_d$ containing $\mathbf{Sob}$ which is adequate and closed with respect to homeomorphisms. Then for any poset $P$, the following two conditions are equivalent:
\begin{enumerate}[\rm (1)]
\item $\Sigma~\!\!\mathbf{K}(\Sigma~\!\!P)$ is a $\mathbf{K}$-space.
\item The $\mathbf{K}$-reflection $(\Sigma~\!\!P)^k$ of $\Sigma~\!\!P$ is a Scott space.
\end{enumerate}
Moreover, when condition (1) holds, the Scott space $\Sigma~\!\!\mathbf{K}(\Sigma~\!\!P)$ with the canonical mapping $\eta_P^{\sigma} : \Sigma~\!\!P \rightarrow \Sigma~\!\!\mathbf{K}(\Sigma~\!\!P)$, $\eta_P^{\sigma}(x)=\cl_{\sigma (P)}\{x\}=\da x$, is a $\mathbf{K}$-reflection of $X$.
\end{theorem}
\begin{proof} First, we show that $\eta_P^{\sigma} : \Sigma~\!\!P \rightarrow \Sigma~\!\!\mathbf{K}(\Sigma~\!\!P)$, $\eta_P^{\sigma}(x)=\cl_{\sigma (P)}\{x\}=\da x$, is continuous. Since $\mathbf{K}$ is adequate, $P_H(\mathbf{K}(\Sigma~\!\!P))$ is $\mathbf{K}$-space, and by Lemma \ref{K-adequate reflective}, $\langle (\Sigma~\!\!P)^k=P_H(\mathbf{K}(\Sigma~\!\!P)), \eta_P\rangle$ is a $\mathbf{K}$-reflection of $X$, where $\eta_P : \Sigma~\!\!P \rightarrow (\Sigma~\!\!P)^k$ is defined by $\eta_P(x)=\cl_{\sigma (P)}\{x\}=\da x$ for each $x\in P$. Whence by Lemma \ref{continuous-ScottCONT-d-space},
$\eta_P^{\sigma} : \Sigma~\!\!P \rightarrow \Sigma~\!\!\mathbf{K}(\Sigma~\!\!P)$ is continuous.

Then by Theorem \ref{K-reflection is Scott space}, conditions (1) and (2) are equivalent, and when condition (1) holds, the Scott space $\Sigma~\!\!\mathbf{K}(\Sigma~\!\!P)$ with the canonical mapping $\eta_P^{\sigma} : \Sigma~\!\!P \rightarrow \Sigma~\!\!\mathbf{K}(\Sigma~\!\!P)$ is a $\mathbf{K}$-reflection of $X$.
\end{proof}

When $\mathbf{K}=\mathbf{Top}_d$ in Theorem \ref{K-reflection is Scott space}, we get the following corollary.

\begin{corollary}\label{d-reflection is Scott} (\cite[Lemma 7.2]{Keimel-Lawson})  Let $X$ be a $T_0$ space. If $\eta_X^{\sigma} : X \rightarrow \Sigma~\!\!\mathbf{d}(X)$, $\eta_X^{\sigma}(x)=\overline{\{x\}}$, is continuous, then the $d$-reflection of $X$ is a Scott space. More precisely, the Scott space $\Sigma~\!\!\mathbf{d}(X)$ with the canonical mapping $\eta_X^{\sigma} : X \rightarrow \Sigma~\!\!\mathbf{d}(X)$ is a $d$-reflection of $X$.
\end{corollary}
\begin{proof} By Lemma \ref{four categories adequate}, $P_H(\mathbf{d}(X))$ is a $d$-space and hence $\mathbf{d}(X)$ (with the order of set inclusion) is a dcpo. Therefore, $\Sigma~\!\!\mathbf{d}(X)$ is a $d$-space. If $\eta_X^{\sigma} : X \rightarrow \Sigma~\!\!\mathbf{d}(X)$ is continuous, then by Theorem \ref{K-reflection is Scott space}, the Scott space $\Sigma~\!\!\mathbf{d}(X)$ with the canonical mapping $\eta_X^{\sigma} : X \rightarrow \Sigma~\!\!\mathbf{d}(X)$ is a $d$-reflection of $X$.
\end{proof}

\begin{corollary}\label{d-reflection of Scott is Scott} (\cite[Corollary 5.9]{XXQ1})
	For any poset $P$, $\mathbf{d}(\Sigma~\!\! P)$ is a dcpo and the Scott space $\Sigma~\!\! \mathbf{d}(\Sigma~\!\! P))$ with the canonical mapping $\eta_{P}: \Sigma~\!\! P\longrightarrow \Sigma~\!\! \mathbf{d}(\Sigma~\!\! P)$, $\eta_{P}(x)=cl_{\sigma(P)}\{x\}$, is a $d$-reflection of $\Sigma~\!\! P$.
\end{corollary}
\begin{proof} By Lemma \ref{K-adequate reflective} and Lemma \ref{four categories adequate}, $P_H(\mathbf{d}(\Sigma~\!\! P))$ with the canonical mapping $\eta_{P}: \Sigma~\!\! P\longrightarrow P_H(\mathbf{d}(\Sigma~\!\! P))$, $x\mapsto cl_{\sigma(P)}\{x\}$, is a $d$-reflection of $\Sigma~\!\! P$. Therefore, $\mathbf{d}(\Sigma~\!\! P)$ is a dcpo and $\eta_{P}: \Sigma~\!\! P\longrightarrow \Sigma~\!\! \mathbf{d}(X)$ is continuous by Lemma  \ref{continuous-ScottCONT-d-space}. By Corollary \ref{d-reflection is Scott}, the Scott space $\Sigma~\!\! \mathbf{d}(\Sigma~\!\! P))$ with the canonical mapping $\eta_{P}: \Sigma~\!\! P\longrightarrow \Sigma~\!\! \mathbf{d}(\Sigma~\!\! P)$ is a $d$-reflection of $\Sigma~\!\! P$.
\end{proof}

For a $T_0$ space $X$ with $\ir_c(X)=\{\overline{\{x\}} : x\in X\}\cup\{X\}$, Theorem \ref{K-reflection is Scott space} can be simplified as the following one.

\begin{theorem}\label{K-reflection is Scott space 1} Let $\mathbf{K}$ be a full subcategory of $\mathbf{Top}_d$ containing $\mathbf{Sob}$ which is adequate and closed with respect to homeomorphisms. Suppose that $X$ is a $T_0$ space for which $\ir_c(X)=\{\overline{\{x\}} : x\in X\}\cup\{X\}$ and $X$ is not a $\mathbf{K}$-space. Consider the following three conditions:
\begin{enumerate}[\rm (1)]
\item $\Sigma~\!\!(\Omega X)_{\top}$ is a $\mathbf{K}$-space.
\item $\zeta_X^{\sigma} : X \rightarrow \Sigma~\!\!(\Omega X)_{\top}$, $\zeta_X^{\sigma}(x)=x$, is continuous.
\item The $\mathbf{K}$-reflection $X^k$ of $X$ is a Scott space.
\end{enumerate}
\noindent Then $ (1)+(2)\Rightarrow (3)$, and $(3)\Rightarrow (1)$. Moreover, when conditions (1) and (2) hold, the Scott space $\Sigma~\!\!(\Omega X)_{\top}$ with the canonical mapping $\zeta_X^{\sigma} : X \rightarrow \Sigma~\!\!(\Omega X)_{\top}$, $\zeta_X^{\sigma}(x)=x$, is a $\mathbf{K}$-reflection of $X$.
\end{theorem}

\begin{proof} (1)+(2) $\Rightarrow$ (3): Since $X$ is not a $\mathbf{K}$-space (and hence not a sober space) and $\ir_c(X)=\{\overline{\{x\}} : x\in X\}\cup\{X\}$, by Lemma \ref{SKIsetrelation} and Lemma  \ref{K-space charac by K-set}, we have $\mathbf{K}(X)=\ir_c(X)=\{\overline{\{x\}} : x\in X\}\cup\{X\}$ and $X\neq \overline{\{y\}}$ for every $y\in X$. Define a mapping $\varphi : (\Omega X)_{\top} \rightarrow \mathbf{K}(X)$ by
$$\varphi (u)=
\begin{cases}
	\overline{\{u\}},& u\in X,\\
    X,& u=\top.
	\end{cases}$$
\noindent Since $X$ is a $T_0$ space, $\varphi : (\Omega X)_{\top} \rightarrow \mathbf{K}(X)$ is a poset isomorphism, and hence induces a homeomorphism from $\Sigma~\!\!(\Omega X)_{\top}$ to $\Sigma~\!\!\mathbf{K}(X)$. It follows from condition (2) that the mapping $\eta_X^{\sigma}=\varphi\circ\zeta_X^{\sigma} : X \rightarrow \Sigma~\!\!\mathbf{K}(X)$ is continuous. Since $\mathbf{K}$ is closed with respect to homeomorphisms, by condition (1), $\Sigma~\!\!\mathbf{K}(X)$ is a $\mathbf{K}$-space. Therefore, by Theorem \ref{K-reflection is Scott space}, $\Sigma~\!\!\mathbf{K}(X)$ with the canonical mapping $\eta_X^{\sigma} : X\rightarrow \Sigma~\!\!\mathbf{K}(X)$ is a $\mathbf{K}$-reflection of $X$, and hence $\Sigma~\!\!(\Omega X)_{\top}$ with the continuous mapping $\zeta_X^{\sigma} : X \rightarrow \Sigma~\!\!(\Omega X)_{\top}$ is a $\mathbf{K}$-reflection of $X$.

(3) $\Rightarrow$ (1): By Theorem \ref{K-reflection is Scott space}, $\Sigma~\!\!\mathbf{K}(X)$ is a $\mathbf{K}$-space. It was shown in the proof of the implication (1)+(2) $\Rightarrow$ (3) that $\varphi : \Sigma~\!\!(\Omega X)_{\top}\rightarrow \Sigma~\!\!\mathbf{K}(X)$, defined by $\varphi(u)=\overline{\{u\}}$ for $u\in X$ and $\varphi(\top)=X$, is a homeomorphism. Since $\mathbf{K}$ is closed with respect to homeomorphisms, $\Sigma~\!\!(\Omega X)_{\top}$ is a $\mathbf{K}$-space.
\end{proof}

By Lemma \ref{poset plus a top remark} (2) and Theorem \ref{K-reflection is Scott space 1}, we get the following corollary.

\begin{corollary}\label{K-reflection is Scott space 3} Let $\mathbf{K}$ be a full subcategory of $\mathbf{Top}_d$ containing $\mathbf{Sob}$ which is adequate and closed with respect to homeomorphisms. Suppose that $P$ is a poset for which $\ir_c(\Sigma~\!\!P)=\{\overline{\{x\}} : x\in P\}\cup\{P\}$ and $\Sigma~\!\!P$ is not a $\mathbf{K}$-space. Then the following two conditions are equivalent:
\begin{enumerate}[\rm (1)]
\item $\Sigma~\!\!P_{\top}$ is a $\mathbf{K}$-space.
\item The $\mathbf{K}$-reflection $(\Sigma~\!\!P)^k$ of $\Sigma~\!\!P$ is a Scott space.
\end{enumerate}
Moreover, when condition (1) holds, the Scott space $\Sigma~\!\!P_{\top}$ with the embedding $i_P : \Sigma~\!\!P \rightarrow \Sigma~\!\!P_{\top}$, $i_P(x)=x$, is a $\mathbf{K}$-reflection of $X$.
\end{corollary}

Now we give some examples and counterexamples related to the $\mathbf{K}$-reflections of $T_0$ spaces (esp., Scott spaces).

\begin{example}\label{K-reflection of Scott N} Let $\mathbf{K}$ be a full subcategory of $\mathbf{Top}_d$ containing $\mathbf{Sob}$ which is adequate and closed with respect to homeomorphisms. Since $\mathbb{N}$ is not a dcpo, $\Sigma~\!\!\mathbb{N}$ is not a $d$-space and hence not a $\mathbf{K}$-space. Clearly, $\ir_c(\Sigma~\!\!\mathbb{N})=\{\overline{\{n\}}=\da n : n\in \mathbb{N}\}\bigcup\{\mathbb{N}\}$. As $\mathbb{N}_{\top}$ is an algebraic lattice, by Proposition \ref{Continuous domain is sober}, $\Sigma~\!\!\mathbb{N}_{\top}$ is a sober space and hence a $\mathbf{K}$-space. By Corollary \ref{K-reflection is Scott space 3}, the $\mathbf{K}$-reflection of $\Sigma~\!\!\mathbb{N}$ is a Scott space. More precisely, $\Sigma~\!\!\mathbb{N}_{\top}$ with the embedding $i_\mathbb{N} : \Sigma~\!\!\mathbb{N} \rightarrow \Sigma~\!\!\mathbb{N}_{\top}$, $i_\mathbb{N}(n)=n$, is a $\mathbf{K}$-reflection of $\Sigma~\!\!\mathbb{N}$.
\end{example}

\begin{example}\label{K-reflection of Scott N plus a and b}
Let $\mathbf{K}$ be a full subcategory of $\mathbf{Top}_d$ containing $\mathbf{Sob}$ which is adequate and closed with respect to homeomorphisms and $P=\mathbb{N}\cup\{a,b\}$. Define a partial order $\leq$ on $P$ as follows:
\begin{enumerate}[\rm (i)]
\item $n<n+1$ for each $n\in \mathbb{N}$,

\item $n<a$ and $n<b$ for all $n\in \mathbb{N}$, and

\item $a$ and $b$ are incomparable.
\end{enumerate}

\noindent Then ${\rm max} (P)=\{a, b\}$ and $P$ is not a dcpo since the chain $\mathbb{N}$ does not have a least upper bound in $P$. So  $\Sigma~\!\!P$ is not a $d$-space and hence not a $\mathbf{K}$-space. It is easy to verify that  $\ir_{c}(\Sigma~\!\!P)=\{\da x:x\in P\}\cup \{\mathbb{N}\}$. Whence by Lemma \ref{SKIsetrelation} and Lemma \ref{K-space charac by K-set}, $\mathbf{K}(\Sigma~\!\!P)=\ir_{c}(\Sigma~\!\!P)=\{\da x:x\in P\}\cup \{\mathbb{N}\}$. Now we show that $\Sigma~\!\!\mathbf{K}(\Sigma~\!\!P)$ is sober. Let $Q=\mathbb{N}\cup \{a,b,c\}$. Define a partial order $\leq_Q$ on $Q$ as follows:
\begin{enumerate}[\rm (a)]
\item for $x, y\in P$, $x\leq_Q y$ iff $x\leq_P y$ in $P$,

\item $n<_Q c$ for all $n\in \mathbb{N}$, and

\item $c<_Q a$ and $c<_Q b$.
\end{enumerate}
\noindent Clearly, $Q$ is an algebraic domain and $K(Q)=\mathbb{N}\cup\{a, b\}$. Define a mapping $\psi : \mathbf{K}(\Sigma~\!\!P)\rightarrow Q$ by
$$\psi (x)=
\begin{cases}
	n,& x=\ua n ~(n\in \mn),\\
    c,& x=\mathbb{N}, \\
    a,& x=\{a\},\\
    b,& x=\{b\}.
	\end{cases}$$
\noindent It is straightforward to verify that $\psi$ is a poset isomorphism, and hence induces a homeomorphism from $\Sigma~\!\!\mathbf{K}(\Sigma~\!\!P)$ to $\Sigma~\!\!Q$. Clearly, $Q$ is a dcpo, $K(Q)=\mathbb{N}\cup\{a, b\}$ and $c=\vee_Q \mathbb{N}$, whence $Q$ is an algebraic domain. By Proposition \ref{Continuous domain is sober}, $\Sigma~\!\!Q$ is sober, and consequently, $\Sigma~\!\!\mathbf{K}(\Sigma~\!\!P)$ is a sober space and hence a $\mathbf{K}$-space. It follows from Theorem \ref{K-reflection is Scott space 2} that the $\mathbf{K}$-reflection of $\Sigma~\!\!P$ is a Scott space. More precisely, $\Sigma~\!\!Q$ with the embedding $i_P=\psi\circ\eta_P^\sigma : \Sigma~\!\!P\rightarrow\Sigma~\!\!Q$, $i_P (x)=x$, is a $\mathbf{K}$-reflection of $\Sigma~\!\!P$.
\end{example}

A poset $P$ is said to be \emph{Noetherian} if it satisfies the \emph{ascending chain condition}: every ascending chain has a greatest member. Clearly, $P$ is Noetherian if{}f every directed set of $P$ has a largest element (equivalently, every ideal of $P$ is principal).

The following two examples show that for a $T_0$ space $X$, condition (1) of Theorem \ref{K-reflection is Scott space} is only a necessary condition but not a sufficient condition for the $\mathbf{K}$-reflection of $X$ to be a Scott space.

\begin{example}\label{Xcof}
	Let $\mathbf{K}$ be a full subcategory of $\mathbf{Top}_w$ containing $\mathbf{Sob}$ which is adequate and closed with respect to homeomorphisms. Let $X$ be a countably infinite set and $X_{cof}$ the space equipped with the \emph{co-finite topology} (the empty set and the complements of finite subsets of $X$ are open). Then
\begin{enumerate}[\rm (a)]
    \item $\mathcal C(X_{cof})=\{\emptyset, X\}\cup X^{(<\omega)}$, $X_{cof}$ is $T_1$ and hence a $d$-space.
    \item $\mk (X_{cof})=2^X\setminus \{\emptyset\}$.
    \item $X_{cof}$ is locally compact and first-countable.
    \item $X_{cof}$ is not well-filtered and hence not a $\mathbf{K}$-space.

Let $\mathcal K=\{X\setminus F: F\in X^{(<\omega)}\}$. Then $\mathcal K$ is a filtered family of saturated compact subsets of $X_{cof}$ and $\bigcap \mathcal K=\emptyset$, but $X\setminus F\neq\emptyset$ for every $ F\in X^{(<\omega)}$. Thus $X_{cof}$ is not well-filtered.

\item $\mathbf{K}(X_{coc})=\ir_c (X_{cof})=\{\{x\} : x\in X\}\cup\{X\}$.

    It is easy to see that $\ir_c (X_{cof})=\{\{x\} : x\in X\}\cup\{X\}$. By Lemma \ref{SKIsetrelation} and Lemma \ref{K-space charac by K-set}, $\mathbf{K}(X_{coc})=\ir_c (X_{cof})=\{\{x\} : x\in X\}\cup\{X\}$.

\item $\Sigma~\!\!\mathbf{K}(X_{coc})$ is sober and hence a $\mathbf{K}$-space.

Clearly, $\mathbf{K}(X_{coc})$ (with the order of set inclusion) is a Noetherian dcpo and hence is an algebraic domain. By Proposition \ref{Continuous domain is sober}, $\Sigma~\!\!\mathbf{K}(X_{coc})$ is sober, whence it is a $\mathbf{K}$-space.

 \item $\eta_{X_{cof}}^{\sigma} :  X_{cof} \rightarrow \Sigma~\!\!\ir_c(X_{cof})$, $x\mapsto \{x\}$, is not continuous.

Let $C\not\in \mathcal O(X_{cof})$. Then $\{\{x\} : x\in C\}\cup \{X\}\in \sigma (\ir_c(X_{cof}))$, but $(\eta_{X_{cof}}^{\sigma})^{-1}(\{\{x\} : x\in C\}\cup \{X\})=C\not\in \mathcal O(X_{cof})$, proving that $\eta_{X_{cof}}^{\sigma} :  X_{cof} \rightarrow \Sigma~\!\!\ir_c(X_{cof})$ is not continuous.

  \item The $\mathbf{K}$-reflection of $X_{cof}$ is not a Scott space. In particular, the well-filtered reflection of of $X_{cof}$ is not a Scott space  and the sobrification of $X_{cof}$ is also not a Scott space.

Assume, on the contrary, that the $\mathbf{K}$-reflection of $X_{cof}$ is a Scott space. Then there is a poset $P$ such that $(X_{cof})^k=P_H(\mathbf{K}(X_{cof}))$ is homeomorphic to $\Sigma~\!\!P$, whence by (e), $\ir_c(X_{cof})=\mathbf{K}(X_{cof})~(=\Omega P_H(\ir_c(X_{cof})))$ and $P~(=\Omega \Sigma~\!\!P)$ are isomorphic. It follows that $\Sigma~\!\!\ir_c(X_{cof})$ and $\Sigma~\!\!P$ are homeomorphic, and consequently, $P_H(\ir_c(X_{cof}))\cong\Sigma~\!\!\ir_c(X_{cof})$. Therefore, $\mathcal O(P_H(\ir_c(X_{cof})))\cong\sigma (\ir_c(X_{cof}))$, and hence $2^\omega=|\sigma (\ir_c(X_{cof}))|=|\mathcal O(P_H(\ir_c(X_{cof})))|\leq |\mathcal O(X_{cof})|=|X^{(<\omega)}|=\omega$, which is a contradiction by Cantor's Theorem (see \cite[III-2.13 Cantor's Theorem]{Levy}). So the $\mathbf{K}$-reflection of $X_{cof}$ is not a Scott space.

\end{enumerate}
\end{example}

\begin{example}\label{WF space sobrification is not Scott}
	Let $X=2^{\mathbb{N}}$ (the set of all subsets of $\mathbb{N}$) and $X_{coc}$ the space equipped with \emph{the co-countable topology} (the empty set and the complements of countable subsets of $X$ are open). Then
\begin{enumerate}[\rm (a)]
    \item $|X|=\mathfrak{c}=2^{\omega}$ (where $\mathfrak{c}=|\mathbb{R}|$ and $\mathbb{R}$ is the set of  all reals) and $X$ is an uncountably infinite set.
    \item $X_{coc}$ is $T_1$ and $\mathcal C(X_{coc})=\{\emptyset, X\}\cup X^{(\leqslant\omega)}$.
    \item $\mk (X_{coc})=X^{(<\omega)}\setminus \{\emptyset\}$.

    Clearly, every finite subset is compact. Conversely, if $C\subseteq X$ is infinite, then $C$ has an infinite countable subset $\{c_n : n\in\mn\}$. Let $C_0=\{c_n : n\in\mn\}$ and $U_m=(X\setminus C_0)\cup \{c_m\}$ for each $m\in \mn$. Then $\{U_n : n\in\mn\}$ is an open cover of $C$, but has no finite subcover. Whence $C$ is not compact. Thus $\mk (X_{coc})=X^{(<\omega)}\setminus \{\emptyset\}$.

    \item $X_{coc}$ is well-filtered.

    To see this suppose that $\{F_d : d\in D\}\subseteq \mk (X_{coc})$ is a filtered family and $U\in \mathcal O(X_{coc})$ with $\bigcap_{d\in D}F_d\subseteq U$. As $\{F_d : d\in D\}$ is filtered and all $F_d$ are finite, $\{F_d : d\in D\}$ has a least element $F_{d_0}$, and hence $F_{d_0}=\bigcap_{d\in D}F_d\subseteq U$, proving that $X_{coc}$ is well-filtered.
    \item $\ir_c(X_{coc})=\{\{x\} : x\in X\}\cup\{X\}$.
    \item $\Sigma~\!\!\ir_c(X_{coc})$ is sober.

    Let $P=\{\{x\} : x\in X\}\cup\{X\}$ with the order of set inclusion. It is easy to see that $P$ is a Notherian dcpo and hence $\Sigma~\!\!P$ is sober by Proposition \ref{Continuous domain is sober}. Clearly, $\sigma (P)=\gamma (P)$ and hence $|\sigma (P)|=|\gamma (P)|=2^{\mathfrak{c}}$.

    \item $\eta_{X_{coc}}^{\sigma} :  X_{coc} \rightarrow \Sigma~\!\!\ir_c(X_{coc})$, $x\mapsto \{x\}$, is not continuous.

Let $C$ be any non-countable proper subset of $X$, that is, $C\not\in \mathcal C(X_{coc})$. Then $\{\{x\} : x\in C\}\in \mathcal C(\Sigma~\!\! (\ir_c(X_{coc})))$, but $(\eta_{X_{coc}}^{\sigma})^{-1}(\{\{x\} : x\in C\})=C\not\in \mathcal C(X_{coc})$, proving that $\eta_{X_{coc}}^{\sigma} :  X_{coc} \rightarrow \Sigma~\!\!\ir_c(X_{coc})$ is not continuous.

  \item The sobrification of $X_{coc}$ is not a Scott space.

Assume, on the contrary, that the sobrification of $X_{coc}$ is a Scott space. Then there is a poset $P$ such that $(X_{coc})^s=P_H(\ir_c(X_{coc}))$ is homeomorphic to $\Sigma~\!\!P$, whence $\ir_c(X_{coc}) ~(=\Omega P_H(\ir_c(X_{coc})))$ and $P~(=\Omega \Sigma~\!\!P)$ are isomorphic. It follows that $\Sigma~\!\!\ir_c(X_{coc})$ and $\Sigma~\!\!P$ are homeomorphic, and consequently, $P_H(\ir_c(X_{coc}))\cong\Sigma~\!\!\ir_c(X_{coc})$. Therefore, $\mathcal O(P_H(\ir_c(X_{coc})))\cong\sigma (\ir_c(X_{coc}))$, and hence $2^{\mathfrak{c}}=|\sigma (\ir_c(X_{coc}))|=|\mathcal O(P_H(\ir_c(X_{coc})))|\leq |\mathcal O(X_{coc})|=(2^{\omega})^\omega=2^{\omega\cdot\omega}=2^\omega=\mathfrak{c}$ (see \cite[III-3.23 Corollary and III-3.29 Proposition]{Levy}), which is a contradiction by Cantor's Theorem. So the sobrification of $X_{coc}$ is not a Scott space.

\end{enumerate}
\end{example}

Let $\mathbb{J}=\mathbb{N}\times (\mathbb{N}\cup \{\omega\})$ with ordering defined by $(j, k)\leq (m, n)$ if{}f $j = m$ and $k \leq n$, or $n =\omega$ and $k\leq m$. $\mathbb{J}$ is a well-known dcpo constructed by Johnstone in \cite{johnstone-81}.

\begin{proposition}\label{K-reflection of the Johnstone space is not a Scott space} Let $\mathbf{K}$ be a full subcategory of $\mathbf{Top}_w$ containing $\mathbf{Sob}$ which is adequate and closed with respect to homeomorphisms. Then the $\mathbf{K}$-reflection of the Johnstone space $\Sigma~\!\!\mathbb{J}$ is not a Scott space. In particular, neither the sobrification nor the well-filtered reflection of $\Sigma~\!\!\mathbb{J}$ is a Scott space.
\end{proposition}

\begin{proof} Clearly, $\mathbb{J}_{max}=\{(n, \omega) : n\in\mn \}$ is the set of all maximal elements of $\mathbb{J}$. By Remark \ref{poset plus a top remark}, $\mathbb{J}_{\top}$ is a dcpo, and $\top$ is the largest element of $\mathbb{J}_{\top}$ and $\{\top\}\in \sigma (\mathbb{J}_{\top})$. The following three conclusions about $\Sigma~\!\!\mathbb{J}$ are known (see, for example, \cite[Example 3.1]{LL} and \cite[Lemma 3.1]{MLZ}):
\begin{enumerate}[\rm (i)]
\item $\ir_c (\Sigma~\!\!\mathbb{J})=\{\cl_{\sigma (\mathbb{J})}{\{x\}}=\da_{\mathbb{J}} x : x\in \mathbb{J}\}\cup \{\mathbb{J}\}$.
\item $\mathsf{K}(\Sigma~\!\!\mathbb{J})=(2^{\mathbb{J}_{max}} \setminus \{\emptyset\})\bigcup \mathbf{Fin}~\!\mathbb{J}$.
\item $\Sigma~\!\!\mathbb{J}$ is not well-filtered and hence not a $\mathbf{K}$-space.
\end{enumerate}
\noindent Whence we have
\begin{enumerate}[\rm (a)]
\item $\ir_c (\Sigma~\!\!\mathbb{J}_\top)=\{\cl_{\sigma (\mathbb{J}_\top)}{\{x\}}=\da_{\mathbb{J}_\top} x : x\in \mathbb{J}_\top\}\cup \{\mathbb{J}\}$ by (i).
\item $\mathsf{K}(\Sigma~\!\!\mathbb{J}_\top)=\{\ua G : G \mbox{~is nonempty and~} G\subseteq \mathbb{J}_{max}\cup\{\top\}\}\bigcup \mathbf{Fin}~\!\mathbb{J}_\top$ by (ii).
\item $\Sigma~\!\!\mathbb{J}_\top$ is not well-filtered and hence not a $\mathbf{K}$-space.

Indeed, let $\mathcal K=\{\ua_{\mathbb{J}_\top} (\mathbb{J}_{max}\setminus F) : F\in (\mathbb{J}_{max})^{(<\omega)}\}$. Then by (b), $\mathcal K\subseteq \mathsf{K}(\Sigma~\!\!\mathbb{J}_\top)$ is a filtered family and $\bigcap\mathcal{K}=\bigcap_{F\in (\mathbb{J}_{max})^{(<\omega)}} \ua_{\mathbb{J}_\top} (\mathbb{J}_{max}\setminus F)=\bigcap_{F\in (\mathbb{J}_{max})^{(<\omega)}} ((\mathbb{J}_{max}\setminus F)\cup\{\top\})=\{\top\}\cup(\mathbb{J}_{max}\setminus \bigcup (\mathbb{J}_{max})^{(<\omega)})=\{\top\}\in \sigma (\mathbb{J}_\top)$, but there is no $F\in (\mathbb{J}_{max})^{(<\omega)}$ with $\ua_{\mathbb{J}_\top} (\mathbb{J}_{max}\setminus F)\subseteq \{\top\}$. Therefore, $\Sigma~\!\!\mathbb{J}_\top$ is not well-filtered. As $\mathbf{K}$ is a full subcategory of $\mathbf{Top}_w$, $\Sigma~\!\!\mathbb{J}_\top$ is not a $\mathbf{K}$-space.

\item The $\mathbf{K}$-reflection $(\Sigma~\!\!\mathbb{J})^k$ of $\Sigma~\!\!\mathbb{J}$ is not a Scott space.

By (iii), (c) and Corollary \ref{K-reflection is Scott space 3}, the $\mathbf{K}$-reflection $(\Sigma~\!\!\mathbb{J})^k$ of $\Sigma~\!\!\mathbb{J}$ is not a Scott space.

\end{enumerate}
\end{proof}

\section{Scott $\mathbf{K}$-completions of posets}

In this section, we give some applications of the results of Section 5 to the Scott $\mathbf{K}$-completions of posets.

The category whose objects are posets and whose morphisms are monotone (i.e., order-preserving) mappings will be
denoted by $\mathbf{Poset}$, and the full subcategory of dcpos by $\mathbf{DCPO}$. Let $\mathbf{Poset}_s$ denote the category of all posets with Scott continuous mappings and $\mathbf{DCPO}_s$ be the full subcategory of dcpos.

\begin{definition}\label{K-DCPOs}
	Let $\mathbf{K}$ be a full subcategory of $\mathbf{Top}_d$ containing $\mathbf{Sob}$. A poset $P$ is called a \emph{Scott} $\mathbf{K}$-\emph{dcpo}, a $\mathbf{K}$-\emph{dcpo} for short, if $\Sigma~\!\!P$ is a $\mathbf{K}$-space. A poset (even a dcpo) $Q$ is said to be a \emph{non}-$\mathbf{K}$ \emph{poset} if $Q$ is not a $\mathbf{K}$-dcpo. Let $\mathbf{K}$-$\mathbf{DCPO}_s$ denote the category of all $\mathbf{K}$-dcpos with Scott continuous mappings.
\end{definition}

$\mathbf{K}$-$\mathbf{DCPO}_s$ is a full subcategory of $\mathbf{DCPO}_s$, and it is a subcategory of $\mathbf{DCPO}$, but not a full subcategory of $\mathbf{DCPO}$.

Clearly, a poset $P$ is a $\mathbf{Top}_d$-dcpo ($\mathbf{d}$-dcpo for short) iff $P$ is a dcpo. For $\mathbf{K}=\mathbf{Top}_w$, the $\mathbf{K}$-\emph{dcpos} are simply called the $\mathbf{WF}$-dcpos and the category $\mathbf{Top}_w$-$\mathbf{DCPO}_s$ is simply denoted as $\mathbf{WF}$-$\mathbf{DCPO}_s$.

\begin{definition}\label{Ks-comp}
	Let $\mathbf{K}$ be a full subcategory of $\mathbf{Top}_d$ containing $\mathbf{Sob}$. A \emph{Scott} $\mathbf{K}$-\emph{completion}, $\mathbf{K}_s$-\emph{completion} for short, of a poset $P$ is a pair $\langle \widetilde{P}, \eta\rangle$ consisting of a $\mathbf{K}$-dcpo $\widetilde{P}$ and a Scott continuous mapping $\eta :P\longrightarrow \widetilde{P}$, such that for any Scott continuous mapping $f: P\longrightarrow Q$ to a $\mathbf{K}$-dcpo $Q$, there exists a unique Scott continuous mapping $\widetilde{f} : \widetilde{P}\longrightarrow Q$ such that $\widetilde{f}\circ\eta=f$, that is, the following diagram commutes.\\
\begin{equation*}
	\centerline{\xymatrix{
		P \ar[dr]_-{f} \ar[r]^-{\eta}
		&\widetilde{P}\ar@{.>}[d]^-{\widetilde{f}}\\
		&Q}}
	\end{equation*}

\end{definition}

For $\mathbf{K}=\mathbf{Top}_d$ (resp., $\mathbf{K}=\mathbf{Top}_w$), the $\mathbf{K}_s$-completion is simply called the $\mathbf{D}_s$-\emph{completion} (resp., $\mathbf{WF}_s$-\emph{completion}).

By a standard argument, $\mathbf{K}_s$-completions, if they exist, are unique up to isomorphism. We use $\mathbf{K}_s(P)$ to denote the $\mathbf{K}_s$-completion of $P$ if it exists. We will use $\mathbf{D}_s(P)$, $\mathbf{WF}_s(P)$ and $\mathbf{Sob}_s(P)$ to denote the $\mathbf{D}_s$-completion, $\mathbf{WF}_s$-completion and $\mathbf{Sob}_s$-completion of $P$, respectively.

\begin{definition}\label{K-comp}
	Let $\mathbf{K}$ be a full subcategory of $\mathbf{Top}_d$ containing $\mathbf{Sob}$. A $\mathbf{K}$-\emph{completion} of a poset $P$ is a pair $\langle \widetilde{P}, \phi\rangle$ consisting of a $\mathbf{K}$-dcpo $\widetilde{P}$ and a monotone mapping $\phi :P\longrightarrow \widetilde{P}$, such that for any monotone mapping $f: P\longrightarrow Q$ to a $\mathbf{K}$-dcpo $Q$, there exists a unique Scott continuous mapping $\widetilde{f} : \widetilde{P}\longrightarrow Q$ such that $\widetilde{f}\circ\phi=f$.
\end{definition}

For $\mathbf{K}=\mathbf{Top}_d$ (resp., $\mathbf{K}=\mathbf{Top}_w$), the $\mathbf{K}$-completion is simply called the $\mathbf{D}$-\emph{completion} (resp., $\mathbf{WF}$-\emph{completion}).

Similarly, $\mathbf{K}$-completions, if they exist, are unique up to isomorphism. We use $\mathbf{K}(P)$ to denote the $\mathbf{K}$-completion of $P$ if it exists. We will use $\mathbf{D}(P)$, $\mathbf{WF}(P)$ and $\mathbf{Sob}(P)$ to denote the $\mathbf{D}$-completion of $P$, $\mathbf{WF}$-completion and $\mathbf{Sob}$-completion of $P$, respectively.

\begin{remark}\label{DCPOs DCPO-completion} The $\mathbf{D}_s$-completion was called the $\mathbf{D}$-completion in \cite[Definition 1]{ZF}. For the sake of distinction, here we call such a completion the $\mathbf{D}_s$-completion and give the $\mathbf{D}$-completion a little different meaning.
\end{remark}

\begin{definition}\label{weak K dcpo} Let $\mathbf{K}$ be a full subcategory of $\mathbf{Top}_d$ containing $\mathbf{Sob}$. A poset $P$ is called a weak $\mathbf{K}$-\emph{dcpo} if there is a $\mathbf{K}$-space such that $P$ is isomorphic to $\Omega X$.
\end{definition}

Clearly, every weak $\mathbf{K}$-\emph{dcpo} is a dcpo, and a poset $P$ is a dcpo iff $P$ is a $\mathbf{d}$-dcpo iff $P$ is a weak $\mathbf{d}$-dcpo. By Proposition \ref{complete lattice Scott compact closed}, every complete lattice is a $\mathbf{WF}$-dcpo and is also a weak $\mathbf{Sob}$-dcpo. The Isbell lattice $L$ constructed in \cite{Isbell}, as a complete lattice, is a weak $\mathbf{Sob}$-dcpo but not a $\mathbf{Sob}$-dcpo.

\begin{theorem}\label{K-completion 2} Let $\mathbf{K}$ be a full subcategory of $\mathbf{Top}_d$ containing $\mathbf{Sob}$ which is adequate and closed with respect to homeomorphisms. For a poset $P$, if $\mathbf{K}(\Sigma~\!\!P)$ is a $\mathbf{K}$-dcpo, then  $\mathbf{K}_s(P)=\mathbf{K}(\Sigma~\!\!P)$ with the canonical mapping $\eta_P : P \rightarrow \mathbf{K}_s(P)$, $\eta_P(x)=\cl_{\sigma (P)}\{x\}=\da x$, is a $\mathbf{K}_s$-completion of $P$.
\end{theorem}
\begin{proof} By Theorem \ref{K-reflection is Scott space 2}, $\Sigma~\!\!\mathbf{K}(\Sigma~\!\!P)$ with the canonical mapping $\eta_{P}: \Sigma~\!\! P\longrightarrow \Sigma~\!\!\mathbf{K}(\Sigma~\!\!P)$, $\eta_P(x)=cl_{\sigma(P)}\{x\}$, is the $\mathbf{K}$-reflection of $\Sigma~\!\! P$. Therefore, by Lemma \ref{Scott-cont1}, $\mathbf{K}_s(P)=\mathbf{K}(\Sigma~\!\!P)$ with the canonical mapping $\eta_P : P \rightarrow \mathbf{K}_s(P)$, $\eta_P(x)=\cl_{\sigma (P)}\{x\}=\da x$, is a $\mathbf{K}_s$-completion of $P$.
\end{proof}

\begin{definition}\label{K-DCPOs 3}
	Let $\mathbf{K}$ be a full subcategory of $\mathbf{Top}_0$. A poset $P$ is called a $S_\mathbf{K}$-\emph{poset} if $\Sigma~\!\!\mathbf{K}(\Sigma~\!\!P)$ is a $\mathbf{K}$-space. Let $\mathbf{S}_{\mathbf{K}}$-$\mathbf{Poset}_s$ denote the category of all $S_{\mathbf{K}}$-posets with Scott continuous mappings.
\end{definition}

Clearly, $\mathbf{S}_{\mathbf{K}}$-$\mathbf{Poset}_s$ is a full subcategory of $\mathbf{Poset}_s$. If $\mathbf{K}$ is a full subcategory of $\mathbf{Top}_d$ containing $\mathbf{Sob}$, then by Lemma \ref{K-space charac by K-set}, every $\mathbf{K}$-dcpo is a $S_\mathbf{K}$-poset, and hence $\mathbf{K}$-$\mathbf{DCPO}_s$ is a full subcategory of $\mathbf{S}_{\mathbf{K}}$-$\mathbf{Poset}_s$.

From Theorem \ref{K-completion 2} we deduce the following result.

\begin{corollary}\label{K-DCPOs is reflective} Let $\mathbf{K}$ be a full subcategory of $\mathbf{Top}_d$ containing $\mathbf{Sob}$ which is adequate and closed with respect to homeomorphisms. Then $\mathbf{K}$-$\mathbf{DCPO}_s$ is reflective in $\mathbf{S}_{\mathbf{K}}$-$\mathbf{Poset}_s$. Therefore, if $\mathbf{K}(\Sigma~\!\!P)$ is a $\mathbf{K}$-dcpo for any poset $P$, then $\mathbf{K}$-$\mathbf{DCPO}_s$ is reflective in $\mathbf{Poset}_s$.
\end{corollary}

\begin{proposition}\label{DCPOs-completion} For a poset $P$, $\mathbf{D}_s(P)=\mathbf{d}(\Sigma~\!\! P)$ with the canonical mapping $\eta_{P}: P\longrightarrow \mathbf{D}(P)$, $\eta_P(x)=cl_{\sigma(P)}\{x\}$, is a $\mathbf{D}_s$-completion of $P$.
\end{proposition}
\begin{proof} Clearly, $\mathbf{Top}_d$ is closed with respect to homeomorphisms. By Lemma \ref{four categories adequate} $\mathbf{Top}_d$ is adequate and. And by Lemma \ref{Scott-cont1} and Corollary \ref{d-reflection of Scott is Scott}, $\mathbf{d}(\Sigma~\!\! P)$ is a dcpo and $\mathbf{D}_s(P)=\mathbf{d}(\Sigma~\!\! P)$ with the canonical mapping $\eta_{P}: P\longrightarrow \mathbf{D}(P)$, $\eta_P(x)=cl_{\sigma(P)}\{x\}$, is a $\mathbf{D}_s$-completion of $P$.
\end{proof}

\begin{corollary}\label{posets-reflection}\emph{(\cite[Corollary 2]{ZF})}
	$\mathbf{DCPO}_s$ is reflective in $\mathbf{Poset}_s$.
\end{corollary}

\begin{remark}\label{DCPOs-completion 1} In \cite{ZF}, using the $D$-topology defined in \cite{ZF} (see also \cite{Keimel-Lawson}), which originates from Wyler \cite{Wyler}, Zhao and Fan proved that for any poset $P$, the $\mathbf{D}_s$-completion of $P$ exists. Proposition \ref{DCPOs-completion} (or Corollary \ref{d-reflection of Scott is Scott}) shows that the $\mathbf{D}_s$-completion of a poset $P$ is essentially the $d$-reflection of Scott space $\Sigma~\!\! P$.
\end{remark}

By Lemma \ref{four categories adequate} and Theorem \ref{K-completion 2}, we get the following two corollaries.

\begin{corollary}\label{Sober completion} For a poset $P$, if $\ir_c(\Sigma~\!\!P)$ is a $\mathbf{Sob}$-dcpo, then $\mathbf{Sob}_s(P)=\ir_c(\Sigma~\!\!P)$ with the canonical mapping $\eta_P : P \rightarrow \mathbf{Sob}_s(P)$, $\eta_P(x)=\cl_c{\sigma (P)}\{x\}=\da x$, is a $\mathbf{Sob}_s$-completion of $P$.
\end{corollary}

\begin{corollary}\label{well-filtered completion} For a poset $P$, if $\mathbf{WF}(\Sigma~\!\!P)$ is a $\mathbf{WF}$-dcpo, then $\mathbf{WF}_s(P)=\mathbf{WF}(\Sigma~\!\!P)$ with the canonical mapping $\eta_P : P \rightarrow \mathbf{WF}_s(P)$, $\eta_P(x)=\cl_{\sigma (P)}\{x\}=\da x$, is a $\mathbf{WF}_s$-completion of $P$.
\end{corollary}

\begin{proposition}\label{K-completion 2+1} Let $\mathbf{K}$ be a full subcategory of $\mathbf{Top}_d$ containing $\mathbf{Sob}$ which is adequate and closed with respect to homeomorphisms. For a non-$\mathbf{K}$ poset $P$, if $\ir_c(\Sigma~\!\!P)=\{\overline{\{x\}} : x\in P\}\cup\{P\}$ and $P_{\top}$ is a $\mathbf{K}$-dcpo, then $\mathbf{K}_s(P)=P_{\top}$ with the canonical mapping $\eta_P : P \rightarrow P_{\top}$, $\eta_P(x)=x$, is a $\mathbf{K}_s$-completion of $P$.
\end{proposition}

\begin{proof} By Corollary \ref{K-reflection is Scott space 3}, $\mathbf{K}_s(P)=\Sigma~\!\!P_{\top}$ with the canonical mapping $\eta_P : P \rightarrow P_{\top}$, $\eta_P(x)=x$, is a $\mathbf{K}_s$-completion of $P$.
\end{proof}

By Corollary \ref{K-reflection is Scott space 3} and Theorem \ref{K-completion 2}, we get the following result.

\begin{corollary}\label{Sober completion 1}  Let $P$ be a poset $P$. If $\ir_c(\Sigma~\!\!P)=\{\overline{\{x\}} : x\in P\}\cup\{P\}$ and  $\Sigma~\!\!P_{\top}$ is a sober space, then $\mathbf{Sob_s}(P)=P_{\top}$ with the canonical mapping $\eta_P : P \rightarrow P_{\top}$, $\eta_P(x)=x$, is a $\mathbf{Sob_s}$-completion of $P$.
\end{corollary}

\begin{corollary}\label{well-filtered completion 1} Let $P$ be a poset $P$. If $\ir_c(\Sigma~\!\!P)=\{\overline{\{x\}} : x\in P\}\cup\{P\}$ and $\Sigma~\!\!P_{\top}$ is a well-filtered space, then $\mathbf{WF}_s(P)=P_{\top}$ with the canonical mapping $\eta_P : P \rightarrow P_{\top}$, $\eta_P(x)=x$, is a $\mathbf{WF}_s$-completion of $P$.
\end{corollary}

\section{$\mathbf{K}$-reflections of Alexandroff spaces}

In the final section, we discuss the $\mathbf{K}$-reflections of Alexandroff spaces and the $\mathbf{K}$-completions of posets. First, it is easy to verify the following result (cf. \cite[Theorem 5.7]{ZhaoHo}).

\begin{proposition}\label{gamma topology is sober}
	For any poset $P$, the following conditions are equivalent:
	\begin{enumerate}[\rm (1)]
		\item $\Gamma~\!\!P$ is sober.
		\item $\Gamma~\!\!P$ is well-filtered.
		\item $\Gamma~\!\!P$ is a $d$-space.
		\item $P$ is Noetherian.
		\item $P$ is a dcpo such that every element of $P$ is compact \emph{(}i.e., $x\ll x$ for all $x\in P$\emph{)}.
		\item $P$ is a dcpo such that $\gamma(P)=\sigma(P)$.
	\end{enumerate}
\end{proposition}

It is straightforward to verify the following lemma.

\begin{lemma}\label{Alexandroff continuous} For a poset $P$, a $T_0$ space $Y$ and a mapping $f : \Gamma~\!\! P \rightarrow Y$, the following conditions are equivalent:
\begin{enumerate}[\rm (1)]
\item $f : \Gamma~\!\! P \rightarrow Y$ is continuous.
\item $f : P \rightarrow \Omega Y$ is monotone.
\item $f : \Gamma~\!\! P \rightarrow \Gamma~\!\!\Omega Y$ is continuous.
\end{enumerate}
\end{lemma}

\begin{lemma}\label{K-set of Alexandroff spaces} Let $\mathbf{K}$ be a full subcategory of $\mathbf{Top}_d$ containing $\mathbf{Sob}$ and $P$ a poset. Then $\mathbf{K}(\Gamma~\!\!P)=\mathrm{Id}~\! P$.
\end{lemma}
\begin{proof} It is straightforward to verify that $\ir(\Gamma~\!\!P)=\mathcal D(P)$ and $\ir_c(\Gamma~\!\!P)=\mathrm{Id}~\!P$ (see, for example, the first paragraph of \cite[Section 3.2]{Heckmann-Keimel 2003}). By Lemma \ref{SKIsetrelation}, $\mathrm{Id}~\! P=\mathcal{D}_c(\Gamma~\!\!P)\subseteq \mathbf{d}(\Gamma~\!\!P)\subseteq\mathbf{K}(\Gamma~\!\!P)\subseteq\mathbf{Sob}(\Gamma~\!\!P)=\ir_c(\Gamma~\!\!P)=\mathrm{Id}~\! P$. Thus $\mathbf{K}(\Gamma~\!\!P)=\mathrm{Id}~\! P$.
\end{proof}

\begin{theorem}\label{K-reflection of Alexandroff space} Let $\mathbf{K}$ be a full subcategory of $\mathbf{Top}_d$ containing $\mathbf{Sob}$ which is adequate and closed with respect to homeomorphisms. Then for any poset $P$, the $\mathbf{K}$-reflection $(\Gamma~\!\!P)^k$ of $\Gamma~\!\!P$ exists and it is a Scott space. More precisely, the Scott space $\Sigma~\!\!\mathrm{Id} P$ with the canonical mapping $\phi_P : \Gamma~\!\!P \rightarrow \Sigma~\!\!\mathrm{Id}~\! P$, $x\mapsto\cl_{\gamma (P)}\{x\}=\da x$, is a $\mathbf{K}$-reflection of $\Gamma~\!\!P$.
\end{theorem}

\begin{proof} By Lemma \ref{K-set of Alexandroff spaces}, $\mathbf{K}(\Gamma~\!\!P)=\mathrm{Id}~\! P$. Since $\mathrm{Id}~\! P$ is an algebraic domain, by Proposition \ref{Continuous domain is sober}, $\Sigma~\!\!\mathrm{Id}~\! P$ is sober and hence a $\mathbf{K}$-space. Clearly, the map $\phi_P : \Gamma~\!\!P \rightarrow \Sigma~\!\!\mathbf{K}(\Gamma~\!\!P)=\Sigma~\!\!\mathrm{Id}~\! P$, $x\mapsto\cl_{\gamma (P)}\{x\}=\da x$, is monotone; whence by Lemma \ref{Alexandroff continuous}, $\phi_P : \Gamma~\!\!P \rightarrow \Sigma~\!\!\mathbf{K}(\Gamma~\!\!P)$ is continuous. Therefore, Conditions (1) and (2) of Theorem \ref{K-reflection is Scott space} hold for the space $X=\Gamma~\!\!P$. By Theorem \ref{K-reflection is Scott space}, the Scott space $\Sigma~\!\!\mathrm{Id} P$ with the canonical mapping $\phi_P : \Gamma~\!\!P \rightarrow \Sigma~\!\!\mathrm{Id}~\! P$ is a $\mathbf{K}$-reflection of $\Gamma~\!\!P$.
\end{proof}

\begin{remark}\label{K-reflection of Alexandroff space 1} We can present a direct proof of Theorem \ref{K-reflection of Alexandroff space}.
\end{remark}

\begin{proof} By Proposition \ref{Continuous domain is sober} and Lemma \ref{K-set of Alexandroff spaces}, $\mathbf{K}(\Gamma~\!\!P)=\mathrm{Id}~\! P$ and $\Sigma~\!\!\mathrm{Id}~\! P$ is sober and hence a $\mathbf{K}$-space since $\mathrm{Id}~\! P$ is an algebraic domain. Clearly, the canonical mapping $\phi_P : \Gamma~\!\!P \rightarrow \Sigma~\!\!\mathbf{K}(\Gamma~\!\!P)=\Sigma~\!\!\mathrm{Id}~\! P$, $x\mapsto\cl_{\gamma (P)}\{x\}=\da x$, is continuous. Now we show that for each $\mathbf{K}$-space $Y$ and each continuous mapping $f : \Gamma ~\!\!P \rightarrow Y$, there is a unique continuous mapping $f^* :  \Sigma~\!\!\mathrm{Id}~\! P\rightarrow Y$ such that $f^*\circ \phi_P=f$, that is, the following diagram commutes.
\begin{equation*}
\centerline{
\xymatrix{\Gamma~\!\!P \ar[dr]_{f} \ar[r]^-{\zeta_{\Gamma~\!\!\!P}}&  \Sigma~\!\!\mathrm{Id}~\!P\ar@{.>}[d]^{f^{*}} & \\
  & Y  & &
   }}
\end{equation*}

We firstly prove the existence of $f^*$. Since $Y$ is a $\mathbf{K}$-space, $Y$ is a $d$-space. Therefore, $\vee E$ exists in $Y$ for each directed subset $E$ of $Y$ (with the specialization order). As $f : \Gamma ~\!\!P \rightarrow Y$ is continuous, $f : P\rightarrow \Omega Y$ is monotone. Define a mapping $f^* :  \Sigma~\!\!\mathrm{Id}~\! P\rightarrow Y$ by $f^*(I)=\vee f(I)$ for each $I\in \mathrm{Id}~\!P$. For every $\{I_d : d\in D\}\in\mathcal D(\mathrm{Id}~\!P)$, we have that $f^*(\bigvee_{\mathrm{Id}~\!P}\{I_d : d\in D\})=f^*(\bigcup_{d\in D}I_d)=\bigvee f(\bigcup_{d\in D}I_d)=\bigvee_{d\in D} \vee f(I_d)=\bigvee_{d\in D}f^*(I_d)$. By Lemma \ref{Scott-cont1} and Lemma \ref{continuous-ScottCONT-d-space}, $f^* :  \Sigma~\!\!\mathrm{Id}~\! P\rightarrow Y$ is continuous. For each $x\in P$, since $f : P\rightarrow \Omega Y$ is monotone, we have $f^*(\phi_P(x))=\vee f(\da x)=f(x)$, proving that $f^*\circ \phi_P=f$.

Now we prove the uniqueness of $f^*$. Suppose that $g : \Gamma ~\!\!P \rightarrow Y$ is another continuous mapping satisfying $g\circ \circ \phi_P=f$. Then for each $I\in \mathrm{Id}~\!P$, by Lemma \ref{Scott-cont1} and Lemma \ref{continuous-ScottCONT-d-space}, we have that $g(I)=g(\bigcup_{x\in I}\da x)=g(\bigvee_{\mathrm{Id}~\!P}\{\da x : x\in I\})=\bigvee_{x\in I}g(\da x)=\bigvee_{x\in I}f(x)=\bigvee f(I)=f^*(I)$, and hence $g=f^*$.

Therefore, $\langle \Sigma~\!\!\mathrm{Id} P, \phi_P : \Gamma~\!\!P \rightarrow \Sigma~\!\!\mathrm{Id}~\! P\rangle$ is a $\mathbf{K}$-reflection of $\Gamma~\!\!P$.
\end{proof}

From Theorem \ref{K-reflection of Alexandroff space} we immediately deduce the following result.

\begin{corollary}\label{K-completions of posets} Let $\mathbf{K}$ be a full subcategory of $\mathbf{Top}_d$ containing $\mathbf{Sob}$ which is adequate and closed with respect to homeomorphisms. Then for any poset $P$, the $\mathbf{K}$-completion of $P$ exists and it is the pair $\langle \mathrm{Id} P, \phi_P \rangle$, where $\phi_P : P \rightarrow \mathrm{Id}~\! P$ is defined by $\phi_P (x)=\da x$ for each $x\in P$.
\end{corollary}

\begin{corollary}\label{K-DCPO is reflective in Poset} Let $\mathbf{K}$ be a full subcategory of $\mathbf{Top}_d$ containing $\mathbf{Sob}$ which is adequate and closed with respect to homeomorphisms. Then $\mathbf{K}$-$\mathbf{DCPO}_s$ is reflective in $\mathbf{Poset}$.
\end{corollary}

Finally, by Lemma \ref{four categories adequate}, Corollary \ref{K-completions of posets} and Corollary \ref{K-DCPO is reflective in Poset}, we have the following three corollaries.

\begin{corollary}\label{sobrification of Alexandroff space} Let $P$ be a poset. Then the $d$-reflection of $\Gamma~\!\!P$, the well-filtered reflection of $\Gamma~\!\!P$ and the sobrification of $\Gamma~\!\!P$ agree. They all are the Scott space $\Sigma~\!\!\mathrm{Id}$ with the canonical mapping $\phi_P : \Gamma~\!\!P \rightarrow\! P$.
\end{corollary}

\begin{corollary}\label{Sob-completions of posets}
Let $P$ be a poset. Then the $\mathbf{D}$-completion of $P$, the $\mathbf{WF}$-completion of $P$ and the $\mathbf{Sob}$-completion of $P$ agree. They all are the pair $\langle \mathrm{Id} P, \phi_P\rangle$, where $\phi_P : P \rightarrow \mathrm{Id}~\! P$ is defined by $\phi_P (x)=\da x$ for each $x\in P$.
\end{corollary}

\begin{corollary}\label{Sob-DCPO is reflective in Poset}  $\mathbf{DCPO}_s$, $\mathbf{WF}$-$\mathbf{DCPO}_s$ and $\mathbf{Sob}$-$\mathbf{DCPO}_s$ all are reflective in $\mathbf{Poset}$.
\end{corollary}

\end{document}